\documentclass[a4paper]{article}

\usepackage[latin1]{inputenc}
\usepackage[english]{babel}

\usepackage{xcolor}
\usepackage{amsmath,amssymb}
\usepackage{graphicx}

\usepackage[notref, notcite, final]{showkeys}

\usepackage{amsthm}
\usepackage{mathtools}
\usepackage{cleveref}
\usepackage{comment}
\usepackage[shortlabels]{enumitem}
\usepackage{fullpage}

\usepackage{dsfont}

\usepackage{cite}

\newcommand{\todo}[1]{}%{\footnote{\textcolor{red}{\bf #1}}}

\newcommand{\vu}{\textcolor{black}{\mathbf{u}}}
\newcommand{\vv}{\textcolor{black}{\mathbf{v}}}

\newcommand{\vb}{\textcolor{black}{\mathbf{b}}}
\newcommand{\ve}{\textcolor{black}{\mathbf{e}}}

\newcommand{\vO}{\textcolor{black}{\mathbf{0}}}
\newcommand{\vx}{\textcolor{black}{\mathbf{x}}}
\newcommand{\vy}{\textcolor{black}{\mathbf{y}}}

\newcommand{\uI}{\textcolor{black}{\boldsymbol{\mathsf{i}}}}
\newcommand{\uJ}{\textcolor{black}{\boldsymbol{\mathsf{j}}}}
\newcommand{\uK}{\textcolor{black}{\boldsymbol{\mathsf{k}}}}

\newcommand{\nn}{\textcolor{black}{\mathbb{N}}}
\newcommand{\zz}{\textcolor{black}{\mathbb{Z}}}
\newcommand{\rr}{\textcolor{black}{\mathbb{R}}}
\newcommand{\cc}{\textcolor{black}{\mathbb{C}}}
\newcommand{\hh}{\textcolor{black}{\mathbb{H}}}
\renewcommand{\SS}{\textcolor{black}{\mathbb{S}}}

\renewcommand{\Re}{\mathrm{Re}}

\newcommand{\id}{\textcolor{black}{\mathcal{I}}}
\newcommand{\boundOP}{\textcolor{black}{\mathcal{B}}}

\newcommand{\dom}{\operatorname{dom}}

\newcommand{\hil}{\textcolor{black}{\mathcal{H}} }

\newcommand{\QMLat}{\textcolor{black}{\mathcal{L}}}
\newcommand{\MM}{\textcolor{black}{\mathfrak{M}}}

\newcommand{\Center}[1]{\textcolor{black}{\mathcal{C}(#1)}}

\newcommand{\VNAlg}{\textcolor{black}{\mathfrak{R}}}

\newcommand{\II}{\textcolor{black}{\mathsf{I}}}
\newcommand{\JJ}{\textcolor{black}{\mathsf{J}}}
\newcommand{\KK}{\textcolor{black}{\mathsf{K}}}

\newcommand{\EHil}[2]{\textcolor{black}{{E}_{#1}(#2)}}

\newcommand{\poincare}{\mathcal{P}}
\newcommand{\aut}{\mathrm{Aut}}

\newcommand{\Borel}{\textcolor{black}{\mathsf{B}}}

\newcommand{\linspan}[1]{\textcolor{black}{\operatorname{span}_{#1}}}

\newcommand{\clos}[2][]{\textcolor{black}{cl_{#1}(#2)}}

\numberwithin{equation}{section}

\theoremstyle{plain}
\newtheorem{theorem}{Theorem}[section]

\newtheorem{corollary}[theorem]{Corollary}
\newtheorem{conjecture}[theorem]{Conjecture}

\theoremstyle{definition}
\newtheorem{definition}[theorem]{Definition}

\theoremstyle{remark}
\newtheorem{remark}[theorem]{Remark}

\crefname{enumi}{}{}
\crefname{enumii}{}{}

\author{Jonathan Gantner
}

\title{On the Equivalence of Complex and Quaternionic Quantum Mechanics}

\begin{document}
\maketitle

\begin{abstract}
Due to the existence of incompatible observables, the propositional calculus of a quantum system does not form a Boolean algebra but an orthomodular lattice. Such lattice can be realised as a lattice of subspaces on a real, complex or quaternionic Hilbert space, which motivated the formulation of real and quaternionic quantum mechanics in addition to the usual complex formulation. It was argued that any real quantum system admits a complex structure that turns it into a complex quantum system and hence real quantum mechanics was soon discarded. Several authors however developed a quaternionic version of quantum mechanics and this version did not seem to be equivalent its standard formulation on a complex Hilbert space.

Motivated by some recently developed techniques from quaternionic operator theory, we conjecture in this article that this not correct and that any quaternionic quantum system is actually simply the quaternionification of a complex quantum system. We then show that this conjecture holds true for quaternionic relativistic elementary systems by applying some recent arguments that were used to show the equivalence of real and complex elementary relativistic systems. Finally, we conclude by discussing how the misconception that complex and quaternionic quantum mechanics are inequivalent arose from assuming the existence of a left multiplication on the Hilbert space, which is physically not justified.
\end{abstract}

Birkhoff and von Neumann argued in \cite{Birkhoff:1936} that the propositional calculus of a quantum system, which captures the system's logical structure, does due to the existence of incompatible observables not form a Boolean algebra but an orthomodular lattice. Such lattice can be realised as a lattice of subspaces on a real, complex or quaternionic Hilbert space and this motivated the formulation of quantum mechanics not only on complex, but also on real or quaternionic Hilbert spaces. The possibility of real quantum mechanics was however discarded after Stueckelberg argued in \cite{Stueckelberg:1960,Stueckelberg:1961} that any such system admits a complex structure that turns it into a complex quantum system. Quaternionic quantum mechanics however attracted some attention and was developed by several authors after its foundational paper \cite{Finkelstein:1962}. The results of their efforts were finally gathered in the monograph \cite{Adler:1995}, but then the interest dropped---also due to a lack of rigorous mathematical techniques for treating quaternionic linear operators, that made developing the theory further difficult.

The theory of quaternionic linear operator was developed during the last ten to fifteen years after its fundamentals techniques were understood. For an overview on this field we refer to \cite{Colombo:2011}, but we stress that even the spectral theorem for unbounded normal quaternionic linear operators has recently been proved in \cite{Alpay:2016}. 

The quaternionic $\hh$ are a number field that extends the complex numbers. Any quaternion is of the form $x = x_0 + \underline{x} = x_0 + x_1e_1 + x_2e_2 + x_3 e_3$ where the generating units satisfy $e_{\ell}^2 = -1$ and $e_{\ell}e_{\kappa} = - e_{\kappa}e_{\ell}$ for $\ell,\kappa\in\{1,2,3\}$ with $\ell \neq \kappa$. Hence, the 
quaternionic multiplication is not commutative and this caused severe difficulties in developing a proper theory of quaternionic linear operators. Any quaternion can be written in the form $x = x_0 + \uI_x \tilde{x}_1$, where a $\tilde{x}_1 = |\underline{x}|$ and $\uI_{x} = \underline{x}/|\underline{x}|$ belongs to the set of imaginary units $\SS = \{\uI\in\hh: \uI^2 = -1\}$. The set  $\cc_{\uI} = \{ x_0 + \uI x_1: x_0,x_1\in\rr\}$ for a fixed imaginary unit $\uI\in\SS$ is an isomorphic copy of the complex number. Furthermore, if $\uJ\in\SS$ with $\uI\perp\uJ$, then $\uI$ and $\uJ$ anti-commute and any quaternions $x\in\hh$ can be written in the form $x = z_1 + z_2 \uJ$ with unique $z_1,z_2\in\cc_{\uI}$. 
 
One approach to operator theory in the quaternionic setting consists in considering the quaternionic right linear space $V_{R}$ as the quaternionification $V_{R} = V_{\uI} \oplus V_{\uI}\uJ$ with $\uI,\uJ\in\SS$ and $\uI\perp\uJ$ of a properly chosen $\cc_{\uI}$-complex subspace $V_{\uI}$ of $V_{R}$. If this subspace is chosen to suit a certain quaternionic linear operator $T$, then this operator is simply the quaternionic linear extension of a complex linear operator $T_{\uI}$ on the complex component space $V_{\uI}$. It is then sufficient to study the complex linear operator $T_{\uI}$  in order to fully understand the quaternionic linear operator $T$ and this can be done using traditional techniques from complex operator theory. This strategy is often not applicable, but in particular for normal operators on quaternionic Hilbert spaces it is very useful. If $T$ is  for instance an operator on a quaternionic Hilbert space $\hil$ and there exists a unitary and anti-selfadjoint operator $\JJ$ on $\hil$ that commutes with $T$, then we can choose $\uI\in\SS$ and define 
\[
\hil_{\JJ,\uI}^{+} := \{ \vv\in\hil: \JJ\vv = \vv\uI\}\quad\text{and}\quad \hil_{\JJ,\uI}^{-} := \{ \vv\in\hil: \JJ\vv = \vv(-\uI)\}.
\]
The sets $\hil_{\JJ,\uI}^{+}$ and $\hil_{\JJ,\uI}^{-}$ are $\cc_{\uI}$-complex Hilbert spaces with the operations and the scalar product they inherit from $\hil$. Furthermore 
\[
\hil = \hil_{\JJ,\uI}^{+}\oplus \hil_{\JJ,\uI}^{-}
\]
and if $\uJ\in\SS$ with $\uI\perp\uJ$, then $\vv\mapsto\vv\uJ$ is a $\cc_{\uI}$-antilinear isometric bijection from $\hil_{\JJ,\uI}^{+}$ to $\hil_{\JJ,\uI}^{-}$ and so any $\vv\in\hil$ can be written as $\vv = \vv_{1} + \vv_{2}\uJ$ with $\vv_{1},\vv_{2}\in\hil_{\JJ,\uI}^{+}$. The operator $T$ leaves the space $\hil_{\JJ,\uI}^{+}$ invariant as $\JJ T(\vv) = T\JJ(\vv) = (T\vv)\uI$ for any $\vv\in\hil_{\JJ,\uI}^{+}$. The $\cc_{\uI}$-linear operator  $T_{\cc_{\uI}} := T|_{\hil_{{\JJ,\uI}^{+}}}$ on $\hil_{\JJ,\uI}^{+}$ has then properties that are analogue to the one of $T$. The spectrum $\sigma(T_{\cc_{\uI}})$ is a subset of $\sigma_{S}(T)\cap\cc_{\uI}$ and $T_{\cc_{\uI}}$ is bounded, normal, (anti-)selfadjoint or unitary if and only if $T$ has these properties. These ideas go back to \cite{Teichmuller:1936} and were used intensively for instance in \cite{Ghiloni:2013}. The above approach is however not suitable for all quaternionic linear operators---an example of an operator that cannot be treated as the quaternionic linear extension of a complex linear operator on a suitable component space is for instance provided in \cite[Example~5.18]{SpectralOperators}.

In this paper we use the fundamental understanding that was gained about quaternionic linear operators in recent years in order to argue that there exists a substantial logical flaw in the current formulation of quaternionic quantum mechanics. In particular, we conjecture that the any quaternionic quantum system can be reduced to complex quantum mechanics on a suitable component space, just as certain quaternionic linear operators can be reduced to complex linear operators on a suitable complex component space as described above. We furthermore show that this conjecture holds true for elementary relativistic systems in the sense of \cite{Moretti:2017}, which seem to be the only type systems for which also the equivalence of real and complex quantum mechanics is shown properly. Finally, we conclude with a section that discusses how the misconception that complex and quaternionic quantum mechanics are not equivalent arose from the wrong assumption that there exists a physically determined quaternionic left multiplication on the Hilbert space that serves as state space.

\section{Preliminary Results: Internal and External Quaternionification}
We start with recalling two methods for constructing complex and quaternionic Hilbert spaces starting from a real Hilbert space, the procedures of {\em internal} and {\em external} complexification and quaternionification as developed by Sharma in \cite{Sharma:1988}.

Let $\hil$ be a real Hilbert space. Then we can construct a complex Hilbert space $\EHil{\cc}{\hil} = \hil\otimes \cc$ from $\hil$ by considering couples of vectors in $\hil$. Precisely, we define
$\EHil{\cc}{\hil} :=\hil^2$, set for $(\vv_1,\vv_2),(\vu_1,\vu_2)\in\EHil{\cc}{\hil}$ and $a = a_0 + i a_1\in\cc$
\begin{align*}
(\vv_1,\vv_2) + (\vu_1,\vu_2) :=& (\vv_1+ \vu_1, \vv_2 + \vu_2)\\
 a (\vv_1,\vv_2) := &(a_1\vv_1 - a_2\vv_2, a_1\vv_2 + a_2\vv_1)
 \end{align*}
and define a complex scalar product on $\EHil{\cc}{\hil}$ as 
\begin{align*}
\left\langle (\vv_1,\vv_2), (\vu_1,\vu_2)\right\rangle_{\EHil{\cc}{\hil}} :=& \langle \vv_1,\vu_1\rangle_{\hil} + \langle \vv_2,\vu_2\rangle_{\hil}  + i \left( \langle \vv_1,\vu_2\rangle_{\hil} - \langle \vv_2,\vu_1\rangle_{\hil} \right).
\end{align*}
This is the standard procedure for complexifying a real Hilbert space and it corresponds to writing the couples $(\vv_1,\vv_2)$ in $\EHil{\cc}{\hil}$ as $\vv_1 + i \vv_2$ and performing formal computations using the structure on $\hil$ and the fact that $i^2 = -1$. We shall use this notation in the following since it is more convenient.
\begin{definition}
Let $\hil$ be a real Hilbert space. We call the complex Hilbert space $\EHil{\cc}{\hil} = \hil\otimes\cc$ the {\em external complexification} of $\hil$. 
\end{definition}
Any operator $T$ on $\hil$ has a unique complex linear extension $T_{\cc}$ to $\EHil{\cc}{\hil}$ that is obtained by componentwise application, namely 
\[
T_{\cc}(\vv_1 + i\vv_2) = T(\vv_1) + i T(\vv_2).
\]
\pagebreak[2]
\begin{theorem}
Let $\hil$ be a real Hilbert space and let $\EHil{\cc}{\hil}$ be its external complexification.
\begin{enumerate}[(i)]
\item The space $\EHil{\cc}{\hil}$ is a complex Hilbert space of the same dimension as $\hil$ and a set of vectors in $\hil$ is an orthonormal basis of $\hil$ if and only if it is an orthonormal basis of $\EHil{\cc}{\hil}$.  
\item If $T:\dom(T)\subset\hil\to\hil$ is an $\rr$-linear operator, then the domain of its complex linear extension $T_{\cc}$ is $\dom(T_{\cc}) = \dom(T) + \dom(T) i$. Furthermore, $T$ is bounded if and only if $T_{\cc}$ is bounded and in this case $\|T \| = \|T_{\cc}\|$. The extension is compatible with the adjoint, that is $(T_{\cc})^* =  T_{\cc}^{*}$ and consequently $T_{\cc}$ is (anti-)\linebreak[1]selfadjoint, normal or unitary on $\EHil{\cc}{\hil}$ if and only if $T$ is (anti-)selfadjoint, normal or unitary on $\hil$. 
\end{enumerate}
\end{theorem}

Similarly, we can construct a quaternionic Hilbert space $\EHil{\hh}{\hil}$ from a real Hilbert space $\hil$. We can choose $\uI_1,\uI_2\in\SS$ with $\uI_1\perp\uI_2$ and set $\uI_0 = 1$ and $\uI_3 = \uI_1\uI_2$, so that any quaternion $a\in\hh$ can be written as $a = \sum_{\ell=0}^3a_{\ell}\uI_{\ell}$. We then set
\[
\EHil{\hh}{\hil} := \hil\otimes\hh  \cong \left\{ \sum_{\ell = 0}^3\vv_{\ell}\uI_{\ell}: \vv_{\ell}\in\hil\right\},
\]
and define for $ \vv = \sum_{\ell = 0}^3\vv_{\ell}\uI_{\ell}$ and $\vu = \sum_{\ell = 0}^3\vu_{\ell}\uI_{\ell}$ in $\EHil{\hh}{\hil}$ and  $a = \sum_{\ell=0}^3a_{\ell}\uI_{\ell}\in\hh$ the operations
\[
\vv + \vu  :=  \sum_{\ell = 0}^3(\vv_{\ell}+\vu_{\ell})\uI_{\ell},\qquad 
\vv a :=  \sum_{\ell, \kappa= 0}^3(a_{\kappa}\vv_{\ell})\uI_{\ell}\uI_{\kappa},
\]
and
\[
\left\langle \vv, \vu \right\rangle_{\EHil{\hh}{\hil}} :=  \sum_{\ell,\kappa = 0}^3\langle\vv_{\ell},\vu_{\kappa}\rangle_{\hil}\uI_{\ell}\uI{\kappa}
\]
This yields a quaternionic right Hilbert space. The choice of $\uI_1$ and $\uI_2$ in this construction is irrelevant since a different choice yields a Hilbert space that is isometrically isomorphic to $\EHil{\hh}{\hil}$. 

\begin{definition}
Let $\hil$ be a real Hilbert space. We call the quaternionic Hilbert space $\EHil{\hh}{\hil} = \hil\otimes\hh$ the {\em external quaternionification} of $\hil$. 
\end{definition}

  Any operator $T$ on $\hil$ has a unique quaternionic linear extension $T_{\hh}$ to $\EHil{\hh}{\hil}$, which is obtained by componentwise application, namely 
\[
T_{\hh}(\vv) =  \sum_{\ell = 0}^3 T(\vv_{\ell})\uI_{\ell} \qquad \text{for } \vv =  \sum_{\ell = 0}^3\vv_{\ell}\uI_{\ell}\in\EHil{\hh}{\hil}. 
\]
\begin{theorem}
Let $\hil$ be a real Hilbert space and let $\EHil{\hh}{\hil}$ be its external complexification.
\begin{enumerate}[(i)]
\item The space $\EHil{\hh}{\hil}$ is a quaternionic Hilbert space of the same dimension as $\hil$ and a set of vectors in $\hil$ is an orthonormal basis of $\hil$ if and only if it is an orthonormal basis of $\EHil{\hh}{\hil}$.  
\item If $T:\dom(T)\subset\hil\to\hil$ is an $\rr$-linear operator, then the domain of its quaternionic linear extension $T_{\hh}$ is $\dom(T_{\hh}) = \dom(T)\oplus\hh$. Furthermore, $T$ is bounded if and only if $T_{\hh}$ is bounded and in this case $\|T\| = \|T_{\hh}\|$. The extension is compatible with the adjoint, that is $(T_{\hh})^* = T_{\hh}^{*}$ and so $T_{\hh}$ is (anti-)\linebreak[2]selfadjoint, normal or unitary on $\EHil{\hh}{\hil}$ if and only if $T$ is (anti-)selfadjoint, normal or unitary on $\hil$. 
\end{enumerate}
\end{theorem}

Finally, we can construct in a similar manner a quaternionic Hilbert space from a complex Hilbert space. If $\hil$ is a complex Hilbert space over $\cc$, then we can choose $\uI,\uJ\in\SS$ with $\uI\perp\uJ$ and identify the imaginary unit $\uI\in\hh$ with the complex imaginary unit $i\in\cc$ so that $\hil$ becomes a Hilbert space over $\cc_{\uI}$. We can then set 
\[
\EHil{\hh}{\hil} := \hil\oplus\hil\uJ \cong \left\{\vv_1 + \vv_2\uJ: \vv_1,\vv_2\in\hil\right\}
\]
and define for $\vv = \vv_1+ \vv_2\uJ$ and $\vu = \vu_1 + \vu_2\uJ$ in $\EHil{\hh}{\hil}$ and $a = a_1  + a_2\uJ\in\hh$ with $a_1,a_2\in\cc_{\uI}$ the operations
\[
\vv + \vu := (\vv_1 + \vu_1) + (\vv_2 + \vu_2)\uJ,\quad \vv a := (\vv_1a_1 - \vv_2\overline{a_2}) + (\vv_1a_2 + \vv_2\overline{a_1})\uJ
\]
and the scalar product
\[
\langle\vv,\vu\rangle_{\EHil{\hh}{\hil}} := \langle \vv_1,\vu_1\rangle + \overline{\langle \vv_2,\vu_2\rangle} + \left(\langle\vv_1,\vu_2\rangle - \overline{ \langle\vv_2,\vu_1\rangle}\right)\uJ.
\]
This yields again a quaternionic right Hilbert space and the choice of $\uJ\in\SS$ with $\uI\perp\uJ$ in this construction is once more irrelevant since each choice yields a quaternionic Hilbert space that is isomorphically isomorphic to $\EHil{\hh}{\hil}$. 
\begin{definition}
Let $\hil$ be a complex Hilbert space. We call the quaternionic Hilbert space $\EHil{\hh}{\hil}$ the {\em external quaternionification} of $\hil$. 
\end{definition}
 Any complex linear operator $T$ on $\hil$  has a unique quaternionic linear extension $T_{\hh}$ to $\EHil{\hh}{\hil}$ that is obtained by component wise application, namely 
\[
T_{\hh}(\vv) = T(\vv_1) + T(\vv_2)\uJ \qquad \text{for } \vv = \vv_1 + \vv_2\uJ. 
\]
\begin{theorem}
Let $\hil$ be a complex Hilbert space and let $\EHil{\hh}{\hil}$ be its external quaternionification.
\begin{enumerate}[(i)]
\item The space $\EHil{\hh}{\hil}$ is a quaternionic Hilbert space of the same dimension as $\hil$ and a set of vectors in $\hil$ is an orthonormal basis of $\hil$ if and only if it is an orthonormal basis of $\EHil{\hh}{\hil}$.  
\item If $T:\dom(T)\subset\hil\to\hil$ is a complex-linear operator, then the domain of its quaternionic linear extension $T_{\hh}$ is $\dom(T_{\hh}) = \dom(T) + \dom(T)\uJ$. Furthermore, $T$ is bounded if and only if $T_{\hh}$ is bounded and in this case $\|T\| = \|T_{\hil}\|$. The extension is compatible with the adjoint, that is $(T_{\hh})^* = T_{\hh}^{*}$ and so $T_{\hh}$ is (anti-)\linebreak[2]selfadjoint, normal or unitary on $\EHil{\hh}{\hil}$ if and only if $T$ is (anti-)selfadjoint, normal or unitary on $\hil$. 
\end{enumerate}
\end{theorem}

External complexification or quaternionification happens by enlarging the underlying vector space and defining a complex or quaternionic structure on the enlarged space. This is always possible. A different strategy is {\em internal complexification} resp. {\em quaternionification}, which happens by defining a complex resp. quaternionic linear structure on the existing space.

Let $\hil$ be a real Hilbert space and let $\JJ$ be a unitary and anti-selfadjoint operator on $\hil$, that is $\JJ^* = \JJ^{-1} = -\JJ$. Then $\JJ^2 = - \id$ and hence we can define the $\JJ$-induced multiplication with complex scalars on $\hil$ as
\[
(a_0 + i a_1) \vv := a_0 \vv + a_1 \JJ \vv \qquad \forall a = a_0 + i a_1\in\cc, \vv\in\hil
\]
and the $\JJ$-induced complex scalar product on $\hil$ as
\[
\langle \vv,\vu\rangle_{\JJ} := \langle\vv,\vu\rangle  - i \langle \vv,\JJ\vu\rangle.
\]
Since 
\[
\langle\vv,\JJ\vv\rangle = \langle\JJ^*\vv,\vv\rangle = - \langle\JJ\vv,\vv\rangle = - \langle\vv,\JJ\vv\rangle,
\]
$\vv$ and $\JJ\vv$ are orthogonal in $\hil$ for any $\vv\in\hil$ and so the norm induced by $\langle\cdot,\cdot\rangle_{\JJ}$ is the norm induced by $\langle\cdot,\cdot\rangle$. 
\begin{definition}
We call the complex Hilbert space $\hil_{\JJ}:=(\hil,\langle\cdot,\cdot\rangle_{\JJ})$ the internal complexification of $\hil$ that is induced by $\JJ$.
\end{definition}
\begin{theorem}\label{IntComplex}
Let $\hil$ be a real Hilbert space and let $\JJ$ be an anti-selfadjoint and unitary operator on $\hil$.
\begin{enumerate}[(i)]
\item The space $\hil_{\JJ}$ is a complex Hilbert space, the dimension of which is half of the dimension of $\hil$, and a subset of $(\vv_n)_{n\in\varLambda}$ of $\hil$ is an orthonormal basis of $\hil_{\JJ}$ if and only if $( \vv_n)_{n\in\varLambda}\cup(\JJ\vv_n)_{n\in\varLambda}$ is an orthonormal basis of $\hil$. (In particular this implies that $\hil$ has even dimension if its dimension is finite.)
\item An $\rr$-linear operator $T:\dom(T)\subset\hil \to \hil$ is complex linear with respect to the $\JJ$-induced structure if and only if $T$ commutes with $\JJ$. Such operator is bounded as an operator on $\hil$ if and only if it is bounded as an operator on $\hil_{\JJ}$ and in this case $\|T\|_{\boundOP(\hil)} = \|T\|_{\boundOP(\hil_{\JJ})}$. Moreover the adjoint $T^*$ of $T$ on $\hil$ is also the adjoint of $T$ on $\hil_{J}$ and hence $T$ is (anti-)selfadjoint, normal or unitary on $\hil$ if and only if it is bounded, (anti-)selfadjoint, normal or unitary on $\hil_{\JJ}$.
\end{enumerate}
\end{theorem}

Similarly, we can define an internal quaternionification of a real Hilbert space. If $\II$ and $\JJ$ are two anti-selfadjoint and unitary operators on $\hil$ with $\II\JJ = - \JJ \II$, then we can choose $\uI,\uJ\in\SS$ with $\uI\perp\uJ$ and set $\uK := \uI\uJ$. We can then define the multiplication of vectors in $\hil$ with a quaternionic scalar $a = a_0 + a_1\uI + a_2\uJ + a_3\uK\in\hil$ from the right as
\[
\vv a := a_0 \vv + a_1\II\vv + a_2\JJ\vv + a_3 \JJ\II\vv
\]
and, with the abbreviation $\Theta$ for the quadruple $(\II,\JJ,\uI,\uJ)$, a quaternionic scalar product as
\[
\langle\vv,\vu\rangle_{\Theta} := \langle\vv,\vu\rangle - \langle\vv, \II \vu\rangle\uI - \langle\vv, \JJ \vu\rangle\uJ - \langle\vv, \JJ\II \vu\rangle\uK.
\]
\begin{definition}
We call the quaternionic Hilbert space $\hil_{\theta} := (\hil,\langle\cdot,\cdot\rangle_{\Theta})$ the internal quaternionification of $\hil$ that is induced by the quadruple $\Theta =(\II,\JJ,\uI,\uJ)$.
\end{definition}
\begin{theorem}
Let $\hil$ be a real Hilbert space and let $\Theta = (\II,\JJ,\uI,\uJ)$ be a quadruple consisting of two anti-selfadjoint and unitary operators on $\hil$ that anticommute and two imaginary units $\uI,\uJ\in\SS$ with $\uI\perp\uJ$.
\begin{enumerate}[(i)]
\item The space $\hil_{\Theta}$ is a quaternionic (right) Hilbert space, the dimension of which is a quarter of the dimension of $\hil$ and a subset  $(\vv_n)_{n\in\varLambda}$ of $\hil$ is an orthonormal basis for $\hil_{\Theta}$ if and only if $( \vv_n,\II\vv_n,\JJ\vv_n,\JJ\II\vv_n)_{n\in\varLambda}$ is a an orthonormal basis for $\hil$. (In particular this implies that the dimension of $\hil$ is a multiple of four if its dimension is finite.)
\item An $\rr$-linear operator $T:\dom(T)\subset\hil \to \hil$ is quaternionic linear with respect to the $\Theta$-induced structure if and only if it commutes with $\II$ and $\JJ$. Such operator is bounded as an operator on $\hil$ if and only if it is bounded as an operator on $\hil_{\Theta}$ and in this case $\|T\|_{\boundOP(\hil)} = \|T\|_{\boundOP(\hil_{\Theta})}$. Moreover the adjoint $T^*$ of $T$ on $\hil$ is also the adjoint of $T$ on $\hil_{\Theta}$ and hence $T$ is (anti-)selfadjoint, normal or unitary on $\hil$ if and only if it is bounded, (anti-)selfadjoint, normal or unitary on $\hil_{\JJ}$.
\end{enumerate}
\end{theorem}

We conclude this section with the discussion how quaternionic resp. complex Hilbert spaces can be considered as quaternionification or complexifications of their subspaces. 

We start with a complex Hilbert space $\hil $. A conjugation $K$ on $\hil $ is an antilinear and norm-preserving mapping from $\hil $ into itself such that $K\circ K  = \id$. Given a conjugation, we can define $\hil_{K}:= (\id + K)(\hil)$. We find that $\hil_{K}$ is an $\rr$-linear subspace of $\hil $ that is even a real Hilbert space with the structure that it inherits from $\hil $ and that furthermore $\EHil{\cc}{\hil _{K}} = \hil_{K}\otimes\cc \cong \hil $. A complex linear operator $T$ on $\hil $ is then the complex linear extension of an operator on $\hil $ if and only if it commutes with $K$, that is if and only if $T\circ K = K\circ T$.

A conjugation exists in any complex Hilbert space. We can for instance choose an orthogonal basis $(\vb_n)_{n\in\varLambda}$ of $\hil $ and define 
\begin{equation}\label{ConjInduc}
K(\vv) = \sum_{n\in\varLambda}\overline{\langle\vb_n,\vv\rangle_{\hil }}\vb_n.
\end{equation}
The subspace $\hil_{K}$ is then precisely the $\rr$-linear span of $(\vb_{n})_{n\in\varLambda}$. Conversely, if we start from a conjugation $K$, then any orthonormal basis of $\hil_{K}$ induces the conjugation $K$ via \eqref{ConjInduc}.

Let now $\hil$ be a quaternionic Hilbert space. If $\JJ$ is a unitary and anti-selfadjoint operator on $\hil$, then we can choose $\uI\in\SS$ and define 
\[
\hil_{\JJ,\uI}^{+} := \{ \vv\in\hil: \JJ\vv = \vv\uI\}\quad\text{and}\quad \hil_{\JJ,\uI}^{-} := \{ \vv\in\hil: \JJ\vv = \vv(-\uI)\}.
\]
The sets $\hil_{\JJ,\uI}^{+}$ and $\hil_{\JJ,\uI}^{-}$ are $\cc_{\uI}$-complex Hilbert spaces with the operations and the scalar product they inherit from $\hil$. Furthermore 
\[
\hil = \hil_{\JJ,\uI}^{+}\oplus \hil_{\JJ,\uI}^{-} = \hil_{\JJ,\uI}^{+}\oplus \hil_{\JJ,\uI}^{+}\uJ
\] 
and so $\EHil{\hh}{\hil_{\JJ,\uI}^{+}} \cong \hil$. An operator $T$ on $\hil$ is the quaternionic linear extension of a $\cc_{\uI}$-linear operator on $\hil_{\JJ,\uI}^{+}$ if and only if $T$ and $\JJ$ commute.

Finally, if $\II,\JJ$ are two anti-selfadjoint and unitary operators on $\hil$ with $\II\JJ = -\JJ\II$, then we can choose $\uI,\uJ\in\SS$ with $\uI\perp\uJ$ and define 
\[
\hil_{\rr} = \{ \vv\in\hil: \II \vv = \uI, \JJ\vv = \vv\uJ\}
\]
and find that $\hil_{\rr}$ is a real Hilbert space such that $\EHil{\hh}{\hil_{\rr}} = \hil$.  An operator $T$ on $\hil$ is the quaternionic linear extension of an $\rr$-linear operator on $\hil_{\rr}$ if and only if $T$ commutes with $\II$ and $\JJ$.

If we consider the left multiplication $\mathcal{L}$ that is generated on $\hil$ by interpreting $\II$ and $\JJ$ as the multiplication with a $\uI$ and $\uJ$, respectively, then 
 $\hil_{\rr}$ is the real Hilbert space of all vectors that commute with any quaternionic scalar and any orthogonal basis $(\vb_{\ell})_{\ell\in\varLambda}$ of $\hil_{\rr}$ generates the left scalar multiplication via
 \[
 a\vv = \sum_{\ell\in\varLambda}\vb_{\ell}\langle\vb_{\ell},\vv\rangle_{\hil}.
 \]
 
 Observe how defining a left multiplication on a quaternionic Hilbert space is the analogue of defining a conjugation on a complex Hilbert space. They both determine a subspace that serves for writing each vector in terms of components in an $\rr$-linear subspace, which is similar to writing the scalars in $\cc$ resp. $\hh$ in terms of their real components.

\section{A Conjecture About the Equivalence of Complex and Quaternionic Quantum Systems }\label{ConjectureSect}
An experimental proposition about a physical system is the statement that the outcome of an experiment belongs to a certain subset of all possible outcomes. The set of all such experimental propositions and their relations determine the logical structure of the system, which is called its propositional calculus. The propositional calculus of a classical mechanical system has the structure of a Boolean algebra. The propositional calculus of a quantum mechanical system on the other hand has a different structure. The distributive identity, which is valid in Boolean algebras, cannot hold in this setting due to the existence of incompatible observables, which cannot be observed simultaneously. 

Birkhoff and von Neumann argued in \cite{Birkhoff:1936} based on some very plausible physical assumptions that the propositional calculus of a quantum mechanical system carries instead the structure of an orthomodular lattice, which initiated the research interest in the field of quantum logics.
\begin{definition}
A partially ordered system $(L,<)$ is called a lattice if for any $x,y\in L$ there exists
\begin{itemize}
\item a meet $x\wedge  y$ such that $x\wedge  y< x$ and $x\wedge  y < y$ and such that $z< x$ and $z< y$ implies $z< x\wedge  y$ and
\item a join $x\vee y$ such that $x< x\vee y$ and $y< x\vee y$ and such that $x< z$ and $y< z$ implies $x\vee y< z$. 
\end{itemize}
A lattice is called bounded if it has a least element $0$ and greatest element $1$ such that $0< x$ and $x< 1$ for all $x\in L$  and a bounded lattice is called orthocomplemented if every element $x\in L$ has a unique orthocomplement $\neg x$ such that 
\[
\neg(\neg x) = x\qquad x\wedge  \neg x = 0,\qquad x\vee \neg x = 1
\]
and such that 
\[
x< y\quad \text{implies}\quad\neg y < \neg x.
\]
A lattice $L$ is called complete if every subset $A\subset L$ has a greatest lower bound $\bigwedge A$ and a least upper bound $\bigvee A$ and it is called $\sigma$-complete if this holds true for any countable subset $A$ of $L$.  Finally, an orthocomplemented lattice is called modular, if it satisfies for all $x,y\in L$ the orthomodular law
\begin{equation}\label{OrthoMod}
\text{if}\quad x< y,\quad\text{then}\quad y = x\vee(\neg x \wedge  y).
\end{equation}
\end{definition}
\begin{remark}
Birkhoff and von Neumann did actually not arrive at an orthomodular lattice, but at a orthocomplement lattice in which the modular identity 
\[
\text{if}\quad x< z,\quad \text{then} \quad x\vee(y\wedge  z) = (x\vee y)\wedge  z
\]
holds. The weaker form \eqref{OrthoMod} is however the version used today. In particular it is the one used in the paper \cite{Moretti:2017}, the argumentation of which we follow in this section.
\end{remark}
Birkhoff and von Neumann showed that such orthomodular lattice can be realised as a lattice of closed subspaces of a Hilbert space over the real numbers, the complex numbers or over the quaternions \cite{Birkhoff:1936}. The relation $<$ corresponds then to the usual subset relation, the operation $\wedge $ to the intersection and the operation $\vee$ to the closed sum of two subspaces and the orthocomplement $\neg$ corresponds to taking the orthogonal complement of a subspace. Equivalently, we can also consider the lattice of orthogonal projections onto these subspaces instead of the subspaces themselves.

The possibility of formulating quantum mechanics on a real Hilbert space was soon discarded due to the analysis in \cite{Stueckelberg:1960,Stueckelberg:1961}. In these papers, Stueckelberg argues that any such quantum system admits an internal complexification---otherwise Heisenberg's uncertainty principle cannot hold. There must exist an imaginary anti-selfadjoint operator $\JJ$ on the real Hilbert space $\hil$, that commutes with any observable and the unitary group that describes the time development of the system. Hence, observables are complex linear operators on the internal complexification $\hil_{\JJ}$ of $\hil$ that is induced by $\JJ$, cf. \Cref{IntComplex} so that one is actually dealing with a complex quantum system. (The analysis in \cite{Stueckelberg:1960,Stueckelberg:1961} is however not correct as \cite{Sharma:1987} showed. Nevertheless, it is still assumed that any real quantum system admits an internal complexification and a formally correct argument at least for elementary relativistic systems is given in \cite{Moretti:2017}.)

Quantum mechanics on a quaternionic Hilbert space $\hil$ on the other hand was developed by several authors starting with \cite{Finkelstein:1962} and it seemed that such formulation of quantum mechanics was not equivalent to the formulation on a complex Hilbert space \cite{Adler:1995}. However, as we shall see in the following, this seems to be a misconception. Instead, we argue that any quaternionic quantum system is the external quaternionification of a complex quantum system on a suitably chosen complex subspace of $\hil$ and that the belief that the two theories are inequivalent arose from a logical mistake that was made from the very beginning of quaternionic quantum mechanics.

Let us consider a quantum system on a quaternionic Hilbert space and let us assume that there exists a unitary and anti-selfadjoint operator $\JJ$ that commutes with every observable of the system and with the unitary semigroup $U(t),t\in\rr,$ that describes the time evolution. In this case, we can choose $\uI\in\SS$ and reduce the quaternionic quantum system to a $\cc_{\uI}$-complex quantum system on the complex subspace 
\[
\hil_{\JJ,\uI}^{+} = \{\vv\in\hil:\JJ\vv = \vv\uI\}.
\]
 Since all observable and all time translations $U(t)$ commute with $\JJ$, they are  quaternionic linear extensions of operators on $\hil_{\JJ,\uI}^{+}$. The spectral measures of observables also commute with the operators $\JJ$ because the observables themselves do. Hence, the range $K$ of such projection, which corresponds to an experimental proposition in the propositional calculus of the  system, is actually the  quaternionification $K = \EHil{\hh}{K_{\uI}} = K_{\uI} \oplus K_{\uI}\uJ$ with $\uJ\in\SS, \uJ\perp\uI$ of a closed complex linear subspaces on $K_{\uI}$ of $\hil_{\JJ,\uI}^{+}$. The projection itself is in turn the quaternionic linear extension of a projection on $\hil_{\JJ,\uI}^{+}$. In particular, this holds true for one-dimensional subspaces in the propositional calculus and which correspond to pure states of the system. Any such subspace $K_0$ is of the form $K_0 = K_{0,\uI}\oplus K_{0,\uI}\uJ$ with a one-dimensional subspace $K_{0,\uI}$ of $\hil_{\uJ}$. In other words
 \[
 K_0 = \linspan{\cc_{\uI}}(\vv) \oplus\linspan{\cc_{\uI}}(\vv)\uJ = \linspan{\hh}\{\vv\}
 \]
 with some $\vv \in K_{0,\uI}$ and hence any pure state of the system can be represented by a vector in $\hil_{\JJ,\uI}^{+}$. Finally, if we represent a state of the system by a vector $\vv\in\hil_{\JJ,\uI}^{+}$, then the time evolution of the system can be entirely described by vector in $\hil_{\JJ,\uI}^{+}$. The state of the system at time $t>0$ is given by $U(t)\vv$, which belongs again to $\hil_{\JJ,\uI}^{+}$ as 
\[
\JJ(U(t)\vv) = U(t)(\JJ\vv) = U(t)(\vv\uI) = (U(t)\vv)\uI.
\]
The quaternionic quantum system on $\hil$ is therefore simply the external quaternionification of a $\cc_{\uI}$-complex quantum system on $\hil_{\JJ,\uI}^{+}$, which contains all the physically relevant information. We conjecture that this relation is always true.
\begin{conjecture}\label{Conj}
Any quaternionic quantum system is the external quaternionification of a complex quantum system on a complex subspace of the underlying quaternionic Hilbert space.
\end{conjecture}

\section{Classification of Elementary Quantum Systems}
We cannot prove \Cref{Conj} for any arbitrary quantum systems, but we are able to show that it holds true at least for elementary relativistic quantum systems. We show this in \Cref{CStructExist} applying the arguments of \cite{Moretti:2017} where the equivalence of real and complex quantum theories are shown for such systems. We furthermore stress that also the equivalence of real and complex quantum mechanics is only known for this type of system because the argumentation of Stueckelberg is not correct as pointed out in~\cite{Sharma:1987a}.

In order to show the equivalence of complex and quaternionic quantum theory in this special case, we shall further formalise the ideas in \Cref{ConjectureSect}. We consider a quantum system and represent its propositional calculus by a lattice $\QMLat$ of orthogonal projections on a real, complex or quaternionic Hilbert space $\hil$. If $\QMLat$ is the lattice of all orthogonal projections on $\hil$, then we write $\QMLat(\hil)$. We recall several important concepts that are shared in all three settings with the aim of defining a proper notion of quantum system. (We follow the summary of results in \cite{Varadarajan:1985,Luders:2006} given in \cite{Moretti:2017} in order to prepare for the arguments in \Cref{CStructExist}). 

\begin{enumerate}[1)]
\item The orthogonal projections in $\QMLat$ are called {\em elementary observables}. Such elementary observables correspond to experimental propositions and they have only two outcomes: 1 if the proposition is true, 0 if it is wrong.

\item  Observables are modelled by possibly unbounded self-adjoint operators on $\hil$. Any self-adjoint operator $A$ on $\hil$ determines via the spectral theorem a unique spectral measure $E_{A}:\Borel(\rr)\to\boundOP(\hil)$ and conversely any such operator is uniquely determined by its spectral measure. Hence, we can consider the spectral measure $E_{A}$ instead of the operator $A$ itself and we henceforth call a spectral measure defined on the Borel sets $\Borel(\rr)$ of $\rr$, the values of which are orthogonal projections in $\QMLat$, an observable of the quantum system. If $A$ is an observable modelled by the spectral measure $E_{A}: \Borel(\rr) \to \QMLat$, then the interpretation of the elementary proposition $E_{A}(\Delta)$ with $\Delta\in\Borel(\rr)$ is that the outcome of the measurement of $A$ belongs to $\Delta$. 

Two observables are said to be {\em compatible} if they are made of mutually commuting orthogonal projections.

\item A {\em quantum state} is a $\sigma$-additive probability measure over the lattice $\QMLat$. More precisely, a quantum state is a map $\mu:\QMLat\to [0,1]$ such that $\mu(\id) = 1$ and such that for any sequence $(E_{\ell})_{\ell\in\nn}$  in $\QMLat$ with $E_{\ell}E_{\kappa} = 0$ for $\ell\neq \kappa$ one has
\[
\mu\left(s\textrm{-} \sum_{\ell\in\nn}E_{\ell} \right) = \sum_{\ell\in\nn}\mu(E_{\ell}),
\]
where $s\textrm{-}\sum_{\ell\in\nn}$ indicates that the series converges in the strong operator topology. The value of $\mu(E)$ is the probability that the outcome of measuring the proposition $E$ equals $1$ if the state of the system is $\mu$.

{\em Pure states} are extremal points in the convex set of probability measures and they are in one-to-one correspondence with one-dimensional rays in the Hilbert space. If $\vv$ is a unit vector in the ray associated with the pure state $\mu$, then $\mu(E) = \| E\mu \|^2$.

\item {\em L\"{u}ders-von Neumann's post measurement axiom} is in this setting formulated in the following way: If the outcome of the ideal measurement of $F \in\QMLat$ in the state $\mu$ is $1$, then the post measurement state is
\[
\mu_{F}(E):= \frac{\mu(F E F)}{\mu(F)},\qquad \forall E\in\QMLat.
\]
\item A {\em symmetry} is an automorphism $h: \QMLat \to \QMLat$ of the lattice of elementary propositions and we shall denote the set of all such automorphisms by $\aut(\QMLat)$. A subclass of symmetries are those induced by unitary (or in the complex case also anti-unitary) operators $U\in\boundOP(\hil)$ by means of $h_{U}(E) := UEU^{-1}$.

\item A {\em continuous symmetry} is a one-parameter group of lattice automorphisms $(h_s)_{s\in\rr}$ such that $s\mapsto \mu(h_s(E))$ is continuous for every $E\in\QMLat$ and every quantum state $\mu$. The time evolution of the system $(\tau_t)_{t\in\rr}$ is a preferred continuous symmetry.

\item A {\em dynamical symmetry} is a continuous symmetry $(h_s)_{s\in\rr}$ that commutes with the time evolution so that $h_{s}\circ\tau_t = \tau_t\circ h_{s}$ for $s,t\in\rr$. 

\end{enumerate}

Since computations with observables are meaningful, a quantum system should permit an algebraic structure. This structure is the one of an von Neumann algebra. 

\begin{definition}
A Banach algebra over $\mathbb{K} = \rr$ or $\mathbb{K} = \cc$ is a Banach space $(\mathcal{A},\|\cdot\|)$ over $\mathbb{K}$ endowed with a bilinear and associative product $\mathcal{A}\times\mathcal{A}\to\mathcal{A}$ such that
\[
\| \vx\vy \| \leq \|\vx\| \|\vy\|\qquad\forall \vx,\vy\in\mathcal{A}.
\] 

A $*$-algebra over $\mathbb{K} = \rr$ or $\mathbb{K} = \cc$ is a Banach algebra over $\mathbb{K}$ endowed with an involution $*: \mathcal{A}\to\mathcal{A}$ such that 
\[
(\vx^*)^* = \vx\quad (\vx\vy)^* = \vy^*\vx^*\quad (a\vx + b\vy)^* = \zeta(a)\vx^* + \zeta(b)\vy^*
\]
for all $\vx,\vy\in\mathcal{A}$ and all $a,b\in\mathbb{K}$, where $\zeta$ is the identity if $\mathbb{K} = \rr$ and the complex conjugation if $\mathbb{K}  = \cc$.

Finally, a $*$-algebra of $\mathbb{K}$ is called a $C^*$-algebra if in addition 
\[
\| \vx^*\vx\| = \| \vx\|^2
\]
and such $C^*$-algebra is called unital if it contains a neutral element $\ve$.
\end{definition}
It is well-known that the space of bounded operators on a complex Hilbert space forms together with the composition and the adjoint conjugation a complex unital $C^*$-algebra and similarly the space of bounded operators on a real Hilbert space forms together with the composition and the adjoint conjugation a real unital $C^*$-algebra. The space of bounded operators on a quaternionic Hilbert space however forms together with the composition and the adjoint conjugation again a real unital $C^*$-algebra and not a quaternionic one. Indeed, if $T$ is a quaternionic right linear operator, then the operator $Ta$ is supposed to act as $(Ta)\vv = T(a\vv)$, which is not meaningful since there is no natural multiplication with quaternionic scalars from the left defined on a quaternionic Hilbert space. Defining $(Ta)\vv := T(\vv a)$ on the other hand does only yield a quaternionic linear operator if $a$ is real so that the set of bounded right linear operators  $\boundOP(\hil)$ on  a quaternionic Hilbert space $\hil$ is only a real Banach space.

\begin{definition}
Let $\MM\subset\boundOP(\hil)$ be a set of bounded operators on a real, complex or quaternionic Hilbert space $\hil$. We define the commutant of $\MM$ as
\[
\MM' := \left\{ T\in\boundOP(\hil) : [T,A]:= TA - AT = 0\quad \text{for all $A\in\MM$}\right\}.
\]
\end{definition}
If $\MM$ is closed under the adjoint conjugation, then $\MM'$ is a $*$-algebra with unit. Since the product in $\boundOP(\hil)$ is continuous, $\MM'$ is closed in the uniform operator topology. Hence, if $\MM$ is closed under the adjoint conjugation, then $\MM'$ is a $C^*$-subalgebra in $\boundOP(\hil)$. One can furthermore easily show that $\MM'$ is closed in both the weak and the strong operator topology.

We furthermore have  $\MM\subset(\MM')' =:\MM''$ and $\MM_1'\subset\MM_2'$ if $\MM_1\supset \MM_2$ so that $\MM' = (\MM'')'$. We can therefore not reach beyond the second commutant by iteration. We recall the following important theorem due to von Neumann, the proof of which can be found in any book about operator algebras, cf. for instance Theorem~5.3.1 in \cite{Kadison:1997} (the proof is only formulated for the real or complex setting, but it also holds in the quaternionic one).
\begin{theorem}\label{VNThm}
Let $\hil$ be a real, complex  or quaternionic Hilbert space and let $\mathcal{A}$ be a unital $*$-sub-algebra of $\boundOP(\hil)$. The following statements are equivalent
\begin{enumerate}[(i)]
\item\label{VN1} $\mathcal{A} = \mathcal{A}''$.
\item\label{VN2} $\mathcal{A}$ is weakly closed.
\item\label{VN3} $\mathcal{A}$ is strongly closed.
\end{enumerate}
Hence, if $\mathcal{C}$ is a unital $*$-subalgebra of $\boundOP(\hil)$, then $\mathcal{C}'' = \clos[w]{\mathcal{C}} = \clos[s]{\mathcal{C}}$, where $ \clos[w]{\mathcal{C}}$ and $\clos[s]{\mathcal{C}}$ denote the closure with respect to the weak and strong operator topology, respectively. 
\end{theorem}

\begin{definition}
A {\em von Neumann algebra} $\VNAlg$ in the space $\boundOP(\hil)$ of bounded operators on a real, complex or quaternionic Hilbert space is a unital $*$-subalgebra of $\boundOP(\hil)$ that satisfies  the three equivalent conditions \cref{VN1,VN2,VN3} in \Cref{VNThm}. The {\em center} $\Center{\VNAlg}$ of $\VNAlg$ is the abelian von Neumann algebra $\Center{\VNAlg} : = \VNAlg\cap\VNAlg'$. 
\end{definition}
\begin{corollary}
If a set $\MM\subset\boundOP(\hil)$ of bounded operators on a real, complex or quaternionic Hilbert space is closed under the adjoint conjugation, then $\MM''$ is the smallest von Neumann-algebra that contains $\MM$. It is called the von Neumann algebra generated by $\MM$.
\end{corollary}
We shall in the following mainly deal with von Neumann-algebras that are irreducible.
\begin{definition}
Let $\hil$ be a real, complex or quaternionic Hilbert space. A family of operators $\mathfrak{A}\subset\boundOP(U)$ is called reducible if there exists a non-trivial closed subspace $K\subset \hil$ such that $A(K)\subset K$ for all $A\in\mathfrak{A}$. The family $\mathfrak{A}$ is called irreducible, if it is not reducible.
\end{definition}
\begin{remark}
If $\mathfrak{A}$ is irreducible, then it is easy to see that 
\[
\{E\in\QMLat(\hil): [E,A] = 0 \quad \forall A\in\mathfrak{A}\} = \{0,\id\}.
\]
 The opposite implication holds true if $\mathcal{A}$ is closed under adjoint conjugation. In this case, we have for any closed subspace $K\subset \hil$ with $A(K)\subset K$ for all $A\in\mathcal{A}$ that $\langle A\vu, \vv\rangle = \langle\vu, A^*\vv\rangle = 0$ for $\vu\in K^{\perp}$ and $\vv\in K$. Hence, also $A(K^{\perp}) \subset K^{\perp}$ for all $A \in\mathcal{A}$. If $\mathfrak{A}$ is reducible, then we can find a nontrivial subspace $K$ and  the orthogonal projection onto $K$ does then belong to $\{E\in\QMLat(\hil): [E,A] = 0 \quad \forall A\in\mathfrak{A}\}$.
\end{remark}

The differences between von Neumann algebras on a quaternionic Hilbert space and von Neumann algebras on a complex Hilbert space are the same as the differences between von Neumann algebras on a real Hilbert space and von Neumann algebras on a complex Hilbert space stated in \cite[Theorem~2.29]{Moretti:2017}. (We do not recall the proof here, because it is the same for the quaternionic and the real case.)
\begin{theorem}\label{VNAlgThm}
Let $\VNAlg$ be a von Neumann  algebra over a real, complex or quaternionic Hilbert space $\hil$, let $\QMLat(\VNAlg)$ be the lattice of orthogonal projectors in $\VNAlg$ and let
\[
\mathfrak{J}(\VNAlg) := \left\{J \in \VNAlg: J^* = -J, -J^2\in\QMLat(\VNAlg)\right\}.
\]
\begin{enumerate}[(i)]
\item A bounded self-adjoint operator $A$ belongs to $\VNAlg$ if and only if the projections of the the spectral measure of $A$ belong to $\VNAlg$.
\item The set $\QMLat(\VNAlg)$ is a complete (in particular $\sigma$-complete) orthomodular sublattice of $\QMLat(\hil)$.
\item\label{VNAlgThm3} $\VNAlg$ is irreducible if and only if $\QMLat(\VNAlg') = \{0,\id\}$.
\item If $\hil$ is a real or quaternionic Hilbert space, then
\begin{enumerate}[(a)]
\item $\QMLat(\VNAlg)''$ contains all selfadjoint operators in $\VNAlg$.
\item $(\QMLat(\VNAlg)\cup\mathfrak{J}(\VNAlg))'' = \VNAlg$
\item $(\QMLat(\VNAlg))''\subsetneq \VNAlg$ if and only if there exists $J\in \mathfrak{J}(\VNAlg)\setminus\QMLat(\VNAlg)''$. 
\end{enumerate}
\item If $\hil$ is a complex Hilbert space, then $\QMLat(\VNAlg)'' = \VNAlg$. 
\end{enumerate}
\end{theorem}
In order to be able to calculate with observables, it seems reasonable to assume that the set of observables of a quantum mechanical system is embedded in a von Neumann algebra. 
\begin{definition}
A real, complex or quaternionic quantum system  is a von Neumann algebra $\VNAlg$ on a real, complex resp. quaternionic Hilbert space $\hil$. 
\end{definition}
\begin{remark}
We call $\VNAlg$ also the von Neumann algebra of observables. The proper observables are precisely the self-adjoint operators whose spectral measures take values in $\VNAlg$ and the lattice of elementary propositions corresponds to the lattice of orthogonal projections in $\QMLat(\VNAlg)$.  If we consider a complex Hilbert space, then the von Neumann algebra $\VNAlg$ of observables is by \Cref{VNAlgThm} precisely the von Neumann algebra that is generated by the lattice $\QMLat(\VNAlg)$, which represents the propositional calculus of the system. 
\end{remark}

Moretti and Oppio argued in \cite{Moretti:2017} that the symmetry under the Poincar\'{e}-group defines on any elementary relativistic quantum system that is defined on a real Hilbert space an up to sign unique unitary and anti-selfadjoint operator $\JJ$. This operator induces an internal complexification of the real Hilbert space $\hil$ that turns the real quantum system into a complex quantum one. We show now that their arguments can also be used  to show that the symmetry under the Poincar\'{e}-group induces also on any elementary relativistic quantum system a complex structure so that the quaternionic quantum system turns out to be the external quaternionification of a complex elementary relativistic quantum system.

We first give a formal definition of the term {\em elementary quantum system} following the arguments of \cite[Section~5]{Moretti:2017} and show that such systems admit a classification that is analogue to the classification of real elementary quantum systems in Theorem~5.3 of \cite{Moretti:2017}. 

An elementary quantum system must not allow super-selection rules---otherwise we could work separately on the super-selection sectors. Mathematically, this condition is expressed by requiring that the center of $\QMLat(\VNAlg)$ is trivial. Furthermore, we assume that there does not exist any non-trivial orthogonal projection in $\VNAlg'$, that is $\VNAlg$ is irreducible. Such projection could be interpreted as an elementary observable of another external system, whereas we want to be the elementary system to be the entire system we are dealing with. Under the assumption that the center of $\QMLat(\VNAlg)$ is trivial, this condition can also be interpreted as the existence of a maximal set of compatible observables.
\begin{definition}
An elementary real, complex or quaternionic quantum system is an irreducible von Neumann algebra $\VNAlg$ on a separable real, complex resp. quaternionic Hilbert space $\hil$.
\end{definition}
\begin{remark}
A complex quantum system is irreducible if and only if one has $\VNAlg' = \{a\id: a\in\cc\}$ or equivalently if and only if $\VNAlg = \boundOP(\hil)$. Just as in the real case, this is not true in the quaternionic setting, cf. \cite[Remark~5.2]{Moretti:2017}.
\end{remark}
The essential tool we need for showing the equivalence of complex and quaternionic quantum systems, is a precise classification of irreducible von Neumann-algebras on a quaternionic Hilbert space. We first recall the corresponding result for irreducible von Neumann-algebras on a real Hilbert space \cite[Theorem~5.3]{Moretti:2017}.
\begin{theorem}\label{VNT-real}
If $\VNAlg$ is an irreducible von Neumann algebra on a real Hilbert space $\hil$, then precisely one of the following statements holds true.
\begin{enumerate}[(i)]
\item\label{VNT-real1} The commutant $\VNAlg'$ of $\VNAlg$ is isomorphic to the real numbers. Precisely, we have
\[
 \VNAlg' = \{ a\id: a\in\rr\}.
 \]
In this case
\[
\VNAlg  = \boundOP(\hil), \quad \Center{\VNAlg} = \{a\id: a\in\rr\} \quad\text{and}\quad \QMLat(\VNAlg) = \QMLat(\hil)
\]
and we call $\VNAlg$ of {\em real-real type}.
\item\label{VNT-real2} The commutant $\VNAlg'$ of $\VNAlg$ is isomorphic to the field of complex numbers. Precisely, we have
 \[
 \VNAlg' = \{ a\id + b\JJ: a,b\in\rr\},
 \]
 where $\JJ$ is an up to sign unique unitary and anti-selfadjoint operator on $\hil$. Any operator in $\VNAlg$ is complex linear on the internal complexification $\hil_{\JJ}$ of $\hil$ induced by $\JJ$ and we have
 \[
\VNAlg  \cong \boundOP\left(\hil_{\JJ}\right), \quad \Center{\VNAlg} = \{a\id + b\JJ : a,b\in\rr\} \quad\text{and}\quad \QMLat(\VNAlg) \cong \QMLat(\hil_{\JJ}).
\]
In this case, we call $\VNAlg$ of {\em real-complex-type}.
\item\label{VNT-real3} The commutant $\VNAlg'$ of $\VNAlg$ is isomorphic to the skew-field of quaternions. Precisely, we have
\[
\VNAlg' = \{ a\id + b \II + c \JJ + d\KK: a,b,c,d\in\rr\},
\]
 where $\II$, $\JJ$, and $\KK$ are mutually anti-commuting unitary and anti-selfadjoint operators on $\hil$ that do not belong to $\VNAlg$ such that $\II\JJ = \KK$. If we choose $\uI,\uJ\in\SS$ with $\uI\perp\uJ$ and set $\Theta = (\II,\JJ,\uI,\uJ)$, then any operator in $\VNAlg$ is quaternionic right linear on the internal quaternionification $\hil_{\Theta}$ of $\hil$ induced by $\Theta$. Moreover, we have
\[
\VNAlg  \cong \boundOP\left(\hil_{\Theta}\right), \quad \Center{\VNAlg} = \{a\id: a\in\rr\} \quad\text{and}\quad \QMLat(\VNAlg) \cong \QMLat(\hil_{\Theta})
\]
and we call $\VNAlg$ of real-quaternionic type.
\end{enumerate}
\end{theorem}
An analogous result holds for irreducible von Neumann algebras on a quaternionic Hilbert space. In this case, we can however not introduce additional structure on $\hil$ by finding an {\em internal} complexification resp. quaternionification of $\hil$ so that $\VNAlg$ consists of all the linear operators on the more structured space. Instead, we can find a subspace with less structure, so that $\VNAlg$ is the {\em external} quaternionification of all bounded linear operators on this subspace.
\pagebreak[2]
\begin{theorem}\label{VNT}
If $\VNAlg$ is an irreducible von Neumann algebra on a quaternionic Hilbert space $\hil$, then precisely one of the following statements hold true.
\begin{enumerate}[(i)]
\item\label{VNT1} The commutant $\VNAlg'$ of $\VNAlg$ is isomorphic to the real numbers. Precisely, we have
\[
 \VNAlg' = \{ a\id: a\in\rr\}.
 \]
In this case
\[
\VNAlg  = \boundOP(\hil), \quad \Center{\VNAlg} = \{a\id: a\in\rr\} \quad\text{and}\quad \QMLat(\VNAlg) = \QMLat(\hil)
\]
and we call $\VNAlg$ {\em proper quaternionic}.
\item\label{VNT2} The commutant $\VNAlg'$ of $\VNAlg$ is isomorphic to the field of complex numbers. Precisely, we have 
 \[
 \VNAlg' = \{ a\id + b\JJ: a,b\in\rr\},
 \]
 where $\JJ$ is an up to sign unique unitary and anti-selfadjoint operator on $\hil$. If we choose $\uI\in\SS$, then $\hil_{\JJ,\uI}^{+} := \{ \vv\in\hil: \JJ\vv = \vv\uI\}$ is a complex Hilbert space over $\cc_{\uI}$ and 
\[
\VNAlg  \cong \boundOP\left(\hil_{\JJ,\uI}^{+}\right), \quad \Center{\VNAlg} = \{a\id + \JJ : a,b\in\rr\} \quad\text{and}\quad \QMLat(\VNAlg) \cong \QMLat(\hil_{\JJ,\uI}^{+}),
\]
where we identify an operator in $\boundOP\left(\hil_{J,\uI}^{+}\right)$ with its quaternionic linear extension to $\hil$. In this case, we call $\VNAlg$ {\em complex-induced}.

\item\label{VNT3} The commutant $\VNAlg'$ of $\VNAlg$ is isomorphic to the skew-field of quaternions. Precisely, we have
\[
\VNAlg' = \{ a\id + b \II + c \JJ + d\KK: a,b,c,d\in\rr\},
\]
 where $\II$, $\JJ$, and $\KK$ are mutually anti-commuting unitary and anti-selfadjoint operators on $\hil$ that do not belong to $\VNAlg$ such that $\II\JJ = \KK$. If we choose $\uI,\uJ\in\SS$ with $\uI\perp\uJ$ and set $\uK:= \uI\uJ$ , then $\hil_{\rr} : = \{\vv\in\hil: \II \vv = \vv\uI,\JJ\vv = \vv\uI\}$ is a real Hilbert space and 
\[
\VNAlg  \cong \boundOP\left(\hil_{\rr}\right), \quad \Center{\VNAlg} = \{a\id: a\in\rr\} \quad\text{and}\quad \QMLat(\VNAlg) \cong \QMLat(\hil_{\rr}),
\]
where we identify an operator in $\boundOP\left(\hil_{\rr}\right)$ with its quaternionic linear extension to~$\hil$. In this case, we call $\VNAlg$ {\em real-induced}.
\end{enumerate}
\end{theorem}
\begin{proof}
If $T\in\VNAlg'$ is self-adjoint, then its spectral measure takes values in $\VNAlg'$. Since $\VNAlg' = \{0,\id\}$ by \cref{VNAlgThm3} in  \Cref{VNAlgThm} because $\VNAlg$ is irreducible, we have $E(\Delta)  = 0$ or $E(\Delta) = \id$ for any $\Delta\in\Borel(\rr)$. Now observe that there exists precisely one number $n_0\in\zz$ such that $E((n_0-1,n_0]) = \id$. If there existed two such numbers $n_0,n_1\in\zz$, then we would have $E((n_0-1,n_0]) + E((n_1-1,n_1]) = \id + \id = 2\id$, which is not an orthogonal projections. If on the other hand $E((n -1, n]) = 0$ for all $n\in\zz$, we would obtain the contradiction
\[
\id\vv = E(\rr)\vv = \sum_{n\in\zz} E((n-1,n])\vv = \sum_{n\in\zz} 0\vv = \vO
\]
for all $\vv\in\hil$. Let hence $\Delta_0 = (a_0,b_0]$ with $a_0 := n-1$ and $b_0 =n$ be such that $E(\Delta_0) = \id$. We define now inductively a sequence of Borel sets $\Delta_n = (a_n,b_n]$ with $E(\Delta_n) = \id$. Precisely, if $\Delta_n = (a_n,b_n]$ with $E(\Delta_n) = \id$ is given, then the same argument as before shows that either \[
E\left(a_n,(a_n + b_n)/2]\right) = \id\quad\text{or}\quad E\left((a_n+b_n)/2, b_n]\right) = \id.
\]
In the first case we set $a_{n+1} = a_n$ and $b_{n + 1} = (a_{n}+b_{n})/2$ and it the latter case we set $a_{n+1} = (a_{n} + b_{n})/2$ and $b_{n+1} = b_{n}$. Then $E(\Delta_{n+1}) = \id$ for $\Delta_{n+1} = (a_{n+1},b_{n+1}]$. Now let $a = \lim_{n\to +\infty}a_n = \lim_{n\to+\infty}b_n$. Due to the continuity of the the spectral measure with respect to monotone limits in the strong operator topology, we find that
\[
E(\{a\}) = E\left(\bigcap_{n\in\nn}\Delta_n\right) = \lim_{n\to\infty} E(\Delta_n) = \id.
\]
Thus $E(\rr\setminus\{a\}) = 0$ and we conclude that
\[
T = \int_{\rr} s\,dE(s) = \int_{\{a\}} s\,dE(s) + \int_{\rr\setminus\{a\}} s\,dE(s) = a E(\{a\}) = a\id.
\]

Let now $T$ be an arbitrary operator in $\VNAlg'$. Then also $T^*\in\VNAlg'$ and 
\[
T = \frac{1}{2}\left(T+T^*\right) + \frac{1}{2}\left(T-T^*\right).
\]
The operator $T_1:=\frac{1}{2}\left(T + T^*\right)$ is a self adjoint operator and belongs to $\VNAlg'$ and hence $T_1 = a\id$ for some $a\in\rr$ by the above argumentation. The operator $T_2 := \frac{1}{2}\left(T- T^*\right)$ on the other hand is anti-self adjoint and so the operator $T_2^2$ is selfadjoint and belongs to $\VNAlg'$. By the above arguments, we find again that there exists some $c\in\rr$ such that $T_2^2 = c\id$. We even have $c\leq 0$ as 
\[
 c\|\vv\|^2 = \langle \vv ,c\id \vv \rangle = \langle \vv , T_2^2\vv\rangle  = - \langle T_2\vv,T_2\vv\rangle = -\|T_2\vv\|^2
 \]
for any $\vv\in\hil$ due to the anti-selfadjointness of $T_2$. We also see that $c = 0$ if and only if $T_2 = 0$ so that in this case $T = \frac{1}{2}\left(T + T^*\right) = a\id$. If $c \neq 0$, then we set $\JJ  := \frac{1}{\sqrt{-c}} T_2$. Then $\JJ ^* = -\JJ $ because $T_2$ is anti-selfadjoint and $\JJ ^2 = \frac{1}{-c} T_2^2 = -\id$. Setting $b = \sqrt{-c}$, we find that
\[
T = a \id + b\JJ 
\]
with $a,b\in\rr$.

Let now $T$ and $S$ be operators in $\VNAlg'$. Then $T = a\id + b\JJ$ and $S = c\id + d\II$ with $a,b,c,d\in\rr$ and two unitary and anti-selfadjoint operators $\II$ and $\JJ$. Since
\begin{align*}
\|T \vv\|^2 =& \langle T\vv, T\vv\rangle = \langle (a\id + b\JJ)\vv, (a\id +b \JJ) \vv\rangle\\
 =& a^2 \langle\vv,\vv\rangle + ba \langle \JJ \vv,\vv\rangle + ab\langle\vv,\JJ\vv\rangle +b^2 \langle \JJ\vv,\JJ\vv\rangle \\ 
 =& a^2 \langle\vv,\vv\rangle - ab \langle  \vv,\JJ\vv\rangle + ab \langle\vv,\JJ\vv\rangle  + b^2 \langle \vv,\vv\rangle = (a^2 + b^2) \|\vv\|^2,
\end{align*}
for any $\vv\in\hil$ and so in particular $\|T\| =\sqrt{a^2 + b^2}$. Similarly, we see that $\|S\vv\|^2 = (c^2 + d^2)\|\vv\|^2$ and $\|S\| = \sqrt{c^2 + d^2}$. Finally, we deduce from these relations that
\[
\| ST\vv \|^2 = (c^2+d^2)\|T\vv\|^2 = (c^2 + d^2)(a^2 + b^2)\|\vv\|^2
\]
and so 
\[
\|ST\| = \sqrt{(c^2 + d^2)(a^2 + b^2)} =  \sqrt{(c^2 + d^2)}\sqrt{(a^2 + b^2)} =\|S\|\|T\|.
\]

The commutant $\VNAlg'$ is therefore a normed real associative algebra with unit such that $\|T S\| = \|T\|\|S\|$ for all $T,S \in \VNAlg'$. By \cite{Urbanik:1960} any such algebra is isomorphic to either the field of real numbers $\rr$, to the field of complex numbers $\cc$ ore the skew-field of quaternions $\hh$. Let $h$ be an isomorphism of $\VNAlg$ to $\rr$, $\cc$ or $\hh$, respectively. 

If $\VNAlg'$ is isomorphic to $\rr$, then simply $\VNAlg' = h^{-1}(\rr) = \{a\id: a\in\rr\}$ and we find that $\VNAlg = \VNAlg'' = \boundOP(\hil)$, because any quaternionic linear operator commutes with any operator $a\id$ with $a\in\rr$. Hence, we also find $\Center{\VNAlg} = \VNAlg\cap\VNAlg' = \{a\id:a\in\rr\}$ and $\QMLat(\VNAlg) = \QMLat(\hil)$. 

If $\VNAlg'$ is isomorphic to $\cc$, then $\VNAlg' = \{ a\id + b\JJ: a,b\in\rr\}$ with $\JJ = h^{-1}(i)$. Let us show that $\JJ$ is unitary and anti-selfadjoint. Since $h$ is an isomorphism, we have $\JJ^2 = h(i^2) = h(-1) = -\id$. Since $\VNAlg'$ is a $*$-algebra, not only $\JJ$ but also the operator $\JJ^*$ and in turn even  $\JJ\JJ^*$ belong to $\VNAlg'$. Since $\JJ\JJ^*$ is selfadjoint, the arguments at the beginning of the proof imply that 
\[
\JJ\JJ^* = a\id
\]
 for some $a\in\rr$. Moreover, $a >0$ because
\[
a\|\vv\|^2 = \langle \vv, a\id\vv\rangle = \langle \vv, \JJ\JJ^*\vv\rangle =  \langle \JJ^*\vv, \JJ^*\vv\rangle = \|\JJ\vv\|^2.
\]
Since $\JJ^2 = -\id$, we have
$\JJ^* = (-\JJ\JJ)\JJ^* = -\JJ (\JJ\JJ^*) = -\JJ a$ and so $\JJ^* = -\frac{1}{a}\JJ.$ Taking the adjoint, we find that $\JJ = - \frac{1}{a}\JJ^*$ and so also $\JJ^* = - a\JJ$.  Finally, the identity $0 = \JJ^* - \JJ^* = \left(a - \frac{1}{a}\right)\JJ$ implies $a = 1$ and so $\JJ^* = -\JJ$. Hence, $\JJ$ is actually unitary and anti-selfadjoint.

An operator $T\in\boundOP(\hil)$ belongs to $\VNAlg = \VNAlg''$ if and only if it commutes with any operator in $\VNAlg'=\{a\id + b\JJ:a,b\in\rr\}$. Since any operator $\boundOP(\hil)$ commutes with real multiples of the identity, an operator commutes with $\VNAlg'$ if and only if it commutes with $\JJ$. This in turn is the case if and only if the operator $T$ is the quaternionic linear extension of an operator in $\boundOP(\hil_{\JJ,\uI}^{+})$ and so $\VNAlg\cong\boundOP(\hil_{\JJ,\uI}^{+})$. In particular, this implies $\JJ\in\VNAlg$, because it is the quaternionic linear extension of the multiplication with $\uI$ on $\hil_{\JJ,\uI}^{+}$, and  $\Center{\VNAlg} = \VNAlg'$ and $\QMLat(\VNAlg) \cong \QMLat(\hil_{\JJ,\uI}^{+})$. 

Finally, if $\VNAlg'$ is isomorphic to $\hh$, then
\[
\VNAlg' = \{ a\id + b\II + c\JJ + d\KK: a,b,c,d\in\rr\}.
\]
with $\II = h^{-1}(e_1)$, $\JJ = h^{-1}(e_2)$ and $\KK = h^{-1}(e_3)$, where $e_1$, $e_2$, and $e_3$ are the generating units of $\hh$. As above, one can see that $\II$, $\JJ$, and $\KK$ are anti-selfadjoint and unitary. Since an operator belongs to $\VNAlg = \VNAlg''$ if and only if it commutes with $\II$ and $\JJ$ and in turn also with $\KK = \II\JJ$, the operators $\II$, $\JJ$ and $\KK$ themselves do not belong to $\VNAlg$ because they anticommute mutually. If we choose $\uI,\uJ\in\SS$ with $\uI\perp\uJ$, then we can define a left-multiplication on $\hil$ by setting $\II\vv = \vv\uI$ and $\JJ\vv = \vv\uJ$ for all $\vv\in\hil$ and turn into a two-sided Banach space. The space $\hil_{\rr}$ consisting of those vectors that commute with all quaternions is then a real Hilbert space.\todo{Justify!} 

An operator belongs to $\VNAlg = \VNAlg''$ if and only if it commutes with any operator in $\VNAlg'$, that is with any operator of the form $a\id + b\II + c\JJ + d\KK$ or---equivalently---with any quaternion when we consider the left multiplication induced by $\II$ and $\JJ$ on $\hil$. These operators are however precisely those that are quaternionic linear extensions of real-linear operators on $\hil_{\rr}$\todo{, cf. \Cref{OPCompBd} and the discussion before}.

\end{proof}

Wigner's theorem states that any symmetry of an elementary complex quantum system can be represented by a unitary linear or an anti-linear anti-unitary operator on $\hil$. Similarly, any symmetry of a real or a quaternionic quantum system can be represented by a unitary linear operator on $\hil$ \cite[Theorem~4.27]{Varadarajan:1985}. This statement is specified for elementary real systems in \cite[Proposition~5.5]{Moretti:2017}, which we want to recall now.
\begin{theorem}\label{SSThmReal}
We consider an elementary system described by an irreducible von Neumann algebra $\VNAlg$ on  a real Hilbert space. If $h$ is a symmetry of the system---that is $h:\QMLat(\VNAlg)\to\QMLat(\VNAlg)$ is a lattice automorphism---then there exists a unitary operator $U:\hil\to\hil$ such that 
\begin{equation}\label{SymUnR}
h(E) = U E U^{-1}\qquad \forall E\in\QMLat(\VNAlg).
\end{equation}
Furthermore the following facts hold true.
\begin{enumerate}[a)]
\item If $\VNAlg$ is of real-real or real-quaternionic type, then $U\in\VNAlg$. 
\item If $\VNAlg$ is real-complex with $\VNAlg' = \{a \id + b\JJ : a,b\in\rr\}$, then $U$ either commutes with $\JJ$ (and hence $U\in\VNAlg$) or it anticommutes with $\JJ$ (and hence $U\notin \VNAlg'$ but $U^2\in\VNAlg$).
\item If $\VNAlg$ is of real-real or real-quaternionic type, then every unitary operator $U$ in $\VNAlg$ defines a symmetry via \eqref{SymUnR}. Similarly, if $\VNAlg$ is of real-complex type, then every unitary operator $U$ that either commutes or anticommutes with $\JJ$ defines a symmetry via \eqref{SymUnR}. Two such unitary operators $U$ and $U'$ define the same symmetry if and only if $U'U^{-1}\in\Center{\VNAlg}$. 
\end{enumerate}
\end{theorem}
Again we find a similar result for elementary quaternionic quantum systems.
\begin{theorem}\label{SSThm}
We consider an elementary system described by an irreducible von Neumann algebra $\VNAlg$ on  a quaternionic Hilbert space $\hil$. If $h$ is a symmetry---that is $h:\QMLat(\VNAlg)\to\QMLat(\VNAlg)$ is a lattice automorphism---then there exists a unitary operator $U:\hil\to\hil$ such that 
\begin{equation}\label{SymUn}
h(E) = U E U^{-1}\qquad \forall E\in\QMLat(\VNAlg).
\end{equation}
Furthermore the following facts hold true.
\begin{enumerate}[a)]
\item If $\VNAlg$ is proper quaternionic or real induced, then $U\in\VNAlg$. 
\item If $\VNAlg$ is complex induced with $\VNAlg' = \{a \id + bJ : a,b\in\rr\}$, then $U$ either commutes with $\JJ$ (and hence $U\in\VNAlg$) or $\JJ$ anticommutes with $\JJ$ (and hence $U\notin \VNAlg'$ but $U^2\in\VNAlg$).
\item If $\VNAlg$ is proper quaternionic or real induced, then every unitary operator $U$ in $\VNAlg$ defines a symmetry via \eqref{SymUn}. Similarly, if $\VNAlg$ is complex induced, then every unitary operator $U$ that either commutes or anticommutes with $J$ defines a symmetry via \eqref{SymUn}. Two such unitary operators $U$ and $U'$ define the same symmetry if and only if $U'U^{-1}\in\Center{\VNAlg}$. \end{enumerate}
\end{theorem}
\begin{proof}
We recall that the lattice $\QMLat(\VNAlg)$ is isomorphic to $\QMLat(\hil)$, to $\QMLat(\hil_{\JJ,\uI}^{+})$ or to $\QMLat(\hil_{\rr})$  because of \Cref{VNT}. Any isomorphism lattice automorphism on $\QMLat(\VNAlg)$ hence induces a lattice automorphism on  $\QMLat(\hil)$,  $\QMLat(\hil_{\JJ,\uI}^{+})$ resp. $\QMLat(\hil_{\rr})$.

If $\QMLat(\VNAlg) = \QMLat(\hil)$, then the quaternionic version of Wigner's theorem, Theorem~4.27 in  \cite{Varadarajan:1985}, implies for any symmetry $h$ the existence of  a bijective function  $S: \hil\to\hil$ with the properties that 
\begin{enumerate}[(A)]
\item\label{JJAC1} $S$ is additive and
\item\label{JJAC2} there exists $q\in\hh$ with $|q| = 1$ such that $S(\vv a) = \vv q^{-1}a q$ and $\langle S\vv, S\vu \rangle = q^{-1} \langle \vv,\vu\rangle q$ for all $\vu,\vv\in\hil$ and all $a\in\hh$ 
\end{enumerate}
such that
\begin{equation}\label{SES}
h(E) = S E S^{-1}.
\end{equation}
Furthermore any function $S_p:\hil\to\hil$ of the form $S_p(\vv) = S\vv p$ with $p\in\hh$ and $|p| = 1$ also satisfies \eqref{SES}. If we choose $p = q^{-1}$, then we obtain a quaternionic right linear unitary operator $U = S_p\in\boundOP(\hil)$ such that \eqref{SymUn} holds true. Furthermore, again by Wigner's theorem, any other unitary operator in $\boundOP(\hil)$ satisfies \eqref{SymUn} if and only if $U'\vv = S_{r}(\vv) = S(\vv)r = U(\vv)p^{-1}r$ for some $r\in\hh$ with $|r| = 1$. An operator of this form is however quaternionic right linear and unitary if and only if $p^{-1}r\in \{\pm 1\}$, which is equivalent to  $U'U^{-1} = \mp \id$ and hence to $U'U^{-1}$ being a unitary operator in $\Center{\VNAlg}$. Finally, Wigner's theorem also implies that any bijective function $S:\hil\to\hil$ that satisfies \cref{JJAC1,JJAC2} induces a symmetry on $\QMLat(\hil)$ via \eqref{SymUn}. Hence, in particular, any unitary operator on $\hil$ induces a symmetry. 

If $\QMLat(\VNAlg) = \QMLat(\hil_{\rr})$, then any symmetry $h$ on $\QMLat(\VNAlg)$ defines a symmetry $h_{\rr}$ on $\QMLat(\hil_{\rr})$ via
\begin{equation}\label{SESR}
h_{\rr}(E_{\rr}) = h(E)|_{\hil_{\rr}},\qquad \text{if}\quad E_{\rr} = E|_{\rr}.
\end{equation}
Hence, the real version of Wigner's theorem \cite[Theorem~4.27]{Varadarajan:1985} implies the existence of a unitary operator $U_{\rr}$ on $\hil_{\rr}$ such that $h_{\rr}(E_{\rr}) = U_{\rr}E_{\rr}U_{\rr}^{-1}$. If we denote the quaternionic linear extension of $U_{\rr}$ to all of $\hil$ by $U$, then $U$ is a unitary operator on $\hil$. For any $E\in\QMLat(\VNAlg)$ we have after setting $E_{\rr} := E|_{\hil_{\rr}}$ that
\[
h(E)|_{\hil_{\rr}} = h_{\rr}(E_{\rr}) = U_{\rr} E_{\rr} U_{\rr}^{-1}.
\]
Extending these operators to  quaternionic linear operators on $\hil$, we find $h(E) = U E U^{-1}$. It follows also from Wigner's theorem that the operator $U_{\rr}$ is unique up to sign so that a unitary operator $U_{\rr}'$ induces $h$ via \eqref{SymUn} if and only if $U_{\rr}'U_{\rr}^{-1} = \pm \id$. Thus, a unitary operator $U'\in\boundOP(\hil)$ induces $h$ if and only of $(U'U)|_{\hil_{\rr}} = U_{\rr}'U_{\rr}^{-1} = \pm \id$, which is equivalent to $U'U = \pm \id$ and in turn to $U'U^{-1}$ being a unitary operator in $\Center{\VNAlg}$. Finally, Wigner's theorem also states that any unitary operator on $\hil_{\rr}$ induces a symmetry on $\boundOP(\hil_{\rr})$. Since the unitary operators in $\VNAlg$ are exactly the operators that are quaternionic linear extensions of unitary operators on $\hil_{\rr}$, any such operator induces a symmetry on $\QMLat(\VNAlg)$ via \eqref{SESR}.

If finally $\QMLat(\VNAlg) = \QMLat(\hil_{\JJ,\uI}^{+})$, then any symmetry $h$ on $\QMLat(\VNAlg)$ defines a symmetry $h_{\cc_{\uI}}$ on $\QMLat(\hil_{\cc_{\uI}^+})$ via
\begin{equation}\label{CCSym12}
h_{\cc_{\uI}}(E_{\cc_{\uI}}) = h(E)|_{\hil_{\JJ,\uI}^{+}},\qquad \text{if}\quad E_{\cc_{\uI}} = E|_{\hil_{\JJ,\uI}}^{+}.
\end{equation}
Hence, the complex linear version of Wigner's theorem \cite[Theorem~4.28]{Varadarajan:1985} implies the existence of a bijective mapping $S:\hil_{\JJ,\uI}^{+}\to\hil_{\JJ,\uI}^{+}$ such that either
\begin{enumerate}[(I)]
\item\label{CCUO1}$S$ is $\cc_{\uI}$-complex linear and $\langle S\vv,S\vu\rangle_{\hil_{\JJ,\uI}^{+}} = \langle\vv,\vu\rangle_{\hil_{\JJ,\uI}^{+}}$ for all $\vu,\vv\in\hil_{\JJ,\uI}^{+}$, i.e. $S$ is a bounded unitary operator on $\hil_{\JJ,\uI}^{+}$ or
\item\label{CCUO2} $S$ is $\cc_{\uI}$-complex anti-linear and $\langle S\vv,S\vu\rangle_{\hil_{\JJ,\uI}^{+}} = \overline{\langle\vv,\vu\rangle_{\hil_{\JJ,\uI}^{+}}}$ for all $\vu,\vv\in\hil_{\JJ,\uI}^{+}$
\end{enumerate}
and such that
\begin{equation}\label{SESCC}
h_{\cc_{\uI}}(E_{\cc_{\uI}}) = S \circ E_{\cc_{\uI}} \circ S^{-1}\qquad \text{for } E_{\cc_{\uI}} \in\QMLat(\hil_{\JJ,\uI}^{+}).
\end{equation}
If \cref{CCUO1} holds true, then the quaternionic linear extension $U$ of $S$ to $\hil$ is a unitary operator on $\hil$ that commutes with $\JJ$ such that \eqref{SymUn} holds true. If on the other hand \cref{CCUO2} holds true, then we can choose $\uJ\in\SS$ with $\uI\perp\uJ$ and find that the operator $U_{\cc_{\uI}}:\hil_{\JJ,\uI}^{+}\to\hil_{\JJ,\uI}^{-}$ given by $U_{\cc_{\uI}}(\vv) = S(\vv)\uJ$ is a $\cc_{\uI}$-linear unitary operator. Indeed, for $\vu,\vv\in\hil$ and $a\in\cc_{\uI}$, we have 
\[
U_{\cc_{\uI}}(\vv a) = S(\vv a)\uJ = S(\vv)\overline{a}\uJ = S(\vv)\uJ a
\]
and
\begin{gather*}
 \langle U_{\cc_{\uI}}\vv, U_{\cc_{\uI}}\vu \rangle_{\hil_{\JJ,\uI}^{-}} = \langle S(\vv)\uJ, S(\vu)\uJ \rangle_{\hil} = -\uJ \langle S(\vv), S(\vu) \rangle_{\hil} \uJ\\
  = -\uJ \langle S(\vv), S(\vu) \rangle_{\hil_{\JJ,\uI}^{+}} \uJ  =  (-\uJ)\overline{\langle \vv, \vu \rangle_{\hil_{\JJ,\uI}^{+}}} \uJ = \langle \vu, \vv \rangle_{\hil_{\JJ,\uI}^{+}}
   \end{gather*}
because $\overline{\langle \vv, \vu \rangle_{\hil_{\JJ,\uI}^{+}}}\in\cc_{\uI}$. Since $S:\hil_{\JJ,\uI}^{+} \to \hil_{\JJ,\uI}^{+}$ and $\vv \mapsto \vv\uJ:\hil_{\JJ,\uI}^{+}\to \hil_{\JJ,\uI}^{-}$ are bijective, also their composition $U_{\cc_{\uI}}$ is bijective. If we write $\vv \in\hil$ as $\vv = \vv_{1} + \vv_{2}\uJ$ with $\vv_{1},\vv_{2}\in\hil_{\JJ,\uI}^{+}$, then the quaternionic linear extension $U$ of $U_{\cc_{\uI}}$ to all of $\hil$ is given by $U(\vv) = U_{\cc_{\uI}}\vv_{1} + U_{\cc_{\uI}}\vv_{2}\uJ$. This operator is obviously also bijective and moreover unitary since
\begin{align*}
&\langle U\vu,U\vv\rangle_{\hil} = \langle U(\vu_1 + \vu_2 \uJ) , U(\vv_1 + \vv_2 \uJ) \rangle_{\hil} \\
=& \langle U_{\cc_{\uI}}\vu_1 , U_{\cc_{\uI}}\vv_1  \rangle_{\hil_{\JJ,\uI}^{-}} +  \langle U_{\cc_{\uI}}\vu_1 , U_{\cc_{\uI}} \vv_2 \rangle_{\hil_{\JJ,\uI}^{-}} \uJ  \\
&- \uJ \langle U_{\cc_{\uI}} \vu_2  , U_{\cc_{\uI}}\vv_1  \rangle_{\hil_{\JJ,\uI}^{-}} - \uJ \langle U_{\cc_{\uI}} \vu_2  , U_{\cc_{\uI}}\vv_2  \rangle_{\hil_{\JJ,\uI}^{-}} \uJ\\
=& \langle \vu_1 , \vv_1  \rangle_{\hil_{\JJ,\uI}^{+}} +  \langle \vu_1 ,  \vv_2 \rangle_{\hil_{\JJ,\uI}^{+}} \uJ  - \uJ \langle  \vu_2  , \vv_1  \rangle_{\hil_{\JJ,\uI}^{+}} - \uJ \langle  \vu_2  , \vv_2  \rangle_{\hil_{\JJ,\uI}^{+}} \uJ  \\
=& \langle \vu_1 + \vu_2 \uJ , \vv_1 + \vv_2 \uJ \rangle_{\hil} = \langle\vu ,\vv\rangle_{\hil}.
\end{align*}
The inverse of $U$ is the quaternionic linear extension of $U_{\cc_{\uI}}^{-1}$, which is given by 

\[
U^{-1}(\vv) = U_{\cc_{\uI}}^{-1}(\vv_1\uJ)(-\uJ) + U_{\cc_{\uI}}^{-1}(\vv_2\uJ).
\]
On the other hand $U_{\cc_{\uI}}^{-1}(\tilde{\vv}) = S^{-1}(\tilde{\vv}(-\uJ))$ for $\tilde{\vv}\in\hil_{\JJ,\uI}^{-}$ and so 
\[
U^{-1}(\vv) = U_{\cc_{\uI}}^{-1}(\vv_1\uJ)(-\uJ) + U_{\cc_{\uI}}^{-1}(\vv_2\uJ) =  S^{-1}(\vv_1)(-\uJ) + S^{-1}(\vv_2).
\]
 For $E\in\QMLat(\VNAlg)$, we therefore find
\begin{align*}
&UEU^{-1} \vv = UE \left( S^{-1}(\vv_1) (-\uJ) + S^{-1}(\vv_2)\right) =\\
=& U\left(E_{\cc_{\uI}}(S^{-1}(\vv_1))(-\uJ) + E_{\cc_{\uI}}S^{-1}(\vv_2)\right)\\
=& U\left(S^{-1} \left(h_{\cc_{\uI}}(E_{\cc_{\uI}})(\vv_1)\right)(-\uJ) + S^{-1}\left(h_{\cc_{\uI}}(E_{\cc_{\uI}})(\vv_2)\right)\right) = \\
=& U_{\cc_{\uI}}\left(S^{-1} \left(h_{\cc_{\uI}}(E_{\cc_{\uI}})(\vv_1)\right)\right)(-\uJ) + U_{\cc_{\uI}}\left(S^{-1} \left(h_{\cc_{\uI}}(E_{\cc_{\uI}})(\vv_2)\right)\right) \\
=& S\left(S^{-1} \left(h_{\cc_{\uI}}(E_{\cc_{\uI}})(\vv_1)\right)\right)\uJ(-\uJ) + S\left(S^{-1}\left(h_{\cc_{\uI}}(E_{\cc_{\uI}})(\vv_2)\right)\right)\uJ\\
= &h_{\cc_{\uI}}(E_{\cc_{\uI}})(\vv_1)( - \uJ)\uJ + h_{\cc_{\uI}}(E_{\cc_{\uI}})(\vv_2)\uJ = h(E)\vv
\end{align*}
and so also in this case \eqref{SymUn} holds true. 

If $U$ is a unitary operator on $\hil$ that commutes with $\JJ$, then its restriction $S := U|_{\hil_{\cc_{\uI}}}$ to $\hil_{\JJ,\uI}^{+}$ is a unitary operator on $\hil_{\JJ,\uI}^{+}$, that is it satisfies \cref{CCUO1}. Hence, Wigner's theorem implies that $S$ induces a symmetry $h_{\cc_{\uI}}$ on $\QMLat(\hil_{\JJ,\uI}^{+})$ via \eqref{SESCC} and so $U$ induces the symmetry $h$ on $\QMLat(\VNAlg)$ that is characterized by 
\eqref{SESCC} via \eqref{SymUn}. If on the other hand $U$ is a unitary operator that anticommutes with $\JJ$, then $U_{\cc_{\uI}}:=U|_{\hil_{\JJ,\uI}^{+}}$ is a unitary operator from $\hil_{\JJ,\uI}^{+}$ to $\hil_{\JJ,\uI}^{-}$. Consequently, $S(\vv):= (U_{\cc_{\uI}}\vv)(- \uJ)$ is a bijective $\cc_{\uI}$-anti-linear mapping from $\hil_{\JJ,\uI}^{+}$ into itself that satisfies 
\begin{gather*}
\langle S(\vv),  S(\vu) \rangle_{\hil_{\JJ,\uI}^{+}} = \langle (U_{\cc_{\uI}}\vv)(-\uJ), (U_{\cc_{\uI}} \vu)(-\uJ) \rangle_{\hil} = \uJ\langle U_{\cc_{\uI}}\vv, U_{\cc_{\uI}} \vu \rangle_{\hil}(-\uJ)\\
 = \uJ \langle U_{\cc_{\uI}}\vv, U_{\cc_{\uI}} \vu \rangle_{\hil_{\JJ,\uI}^{-}}(-\uJ) = \uJ\langle \vv,  \vu \rangle_{\hil_{\JJ,\uI}^{+}}(-\uJ) = \overline{\langle \vv,  \vu \rangle_{\hil_{\JJ,\uI}^{+}}}
 \end{gather*}
 for any $\vu,\vv\in\hil_{\JJ,\uI}^{+}$. Hence, $S$ satisfies \cref{CCUO2}. Wigner's theorem implies again that $S$ induces a symmetry $h_{\cc_{\uI}}$ on $\QMLat(\hil_{\JJ,\uI}^{+})$ via \eqref{SESCC}. If $E$ is the quaternionic linear extension of $E_{\cc_{\uI}}$ to all of $\hil$, then the quaternionic linear extension of $h_{\cc_{\uI}}(E_{\cc_{\uI}}) = S\circ E_{\cc_{\uI}}\circ S^{-1}$ to all of $\hil$ is due to $S^{-1}(\vv) = U_{\cc_{\uI}}^{-1}(\vv\uJ)$ given by
 \begin{align*}
& h_{\cc_{\uI}}(E_{\cc_{\uI}})(\vv) =  S\circ E_{\cc_{\uI}}\circ S^{-1}(\vv_1) +  S\circ E_{\cc_{\uI}}\circ S^{-1}(\vv_2)\uJ =\\
=& S\circ E_{\cc_{\uI}}\left(U_{\cc_{\uI}}^{-1}(\vv_1\uJ)\right) +  S\circ E_{\cc_{\uI}}\left(U_{\cc_{\uI}}^{-1}(\vv_2\uJ)\right)\uJ\\
=& S\left( EU^{-1}(\vv_1\uJ)\right) +  S\left(EU^{-1}(\vv_2\uJ)\right)\uJ\\
=& U_{\cc_{\uI}}\left( EU^{-1}(\vv_1\uJ)(-\uJ)\right) +  U_{\cc_{\uI}}\left(EU^{-1}\vv_2\uJ(-\uJ)\right)\uJ \\
=& UEU^{-1}\vv_1 + UEU^{-1}\vv_2\uJ = UEU^{-1}\vv.
 \end{align*}
We conclude that $U$ induces the symmetry $h$ on $\QMLat(\VNAlg)$ that is characterized by  \eqref{SESCC} via~\eqref{SymUn}.

Finally, Wigner's theorem also states that two bijective functions $S$ and $S'$ that satisfy \cref{CCUO1} or \cref{CCUO2} induce the same symmetry $h_{\cc_{\uI}}$ on $\boundOP(\hil_{\JJ,\uI}^{+})$ if and only if $S' = \alpha S$ with $\alpha\in\cc_{\uI}$ and $|\alpha| = 1$. In particular either $S$ and $S'$ both satisfy \cref{CCUO1} or they both satisfy \cref{CCUO2}. Now observe that $U$ is a unitary operator on $\hil$ that commutes with $\JJ$ if and only if $S = U|_{\hil_{\JJ,\uI}^{+}}$ satisfies \cref{CCUO1} and that $U$ is a unitary operator that anticommutes with $\JJ$  if and only if the operator $S\vv = U|_{\hil_{\JJ,\uI}^{+}}\vv\uJ$ satisfies \cref{CCUO2}. Since the a unitary operator $U$ that commutes or anticommutes with $\JJ$ induces a symmetry $h$ on $\QMLat(\VNAlg)$ if and only if the respective operator $S$ induces the symmetry $h_{\cc_{\uI}}$ determined by \eqref{SESCC} on $\QMLat(\hil_{\JJ,\uI}^{+})$, we find that two unitary operators $U$ and $U'$ that induce the same symmetry $h$ either both commute or both anticommute with $\JJ$. 

In the first case, the respective operators $S$ and $S'$ are simply the restrictions $S = U|_{\hil_{\JJ,\uI}^{+}}$ and $S' = U|_{\hil_{\JJ,\uI}^{+}}'$. We find due to Wigner's theorem that $U$ and $U'$ induce the same symmetry $h$ if and only if $S = S'\alpha$ with $\alpha\in\cc_{\uI}$ and $|\alpha| =1$, or equivalently
\[
(U'U^{-1})|_{\hil_{\JJ,\uI}^{+}} = S'S^{-1} = \alpha SS^{-1} = \alpha\id.
\]
Since the quaternionic linear extension of the multiplication $\alpha \id$ with the complex number $\alpha = \alpha_0 + \uI \alpha_1\in\cc_{\uI}$ on $\hil_{\JJ,\uI}^{+}$ to all of $\hil$ is the operator $\alpha_0 \id + \alpha_{1} J \in\Center{\VNAlg}$, we find that $U$ and $U'$ induce the same symmetry on $\QMLat(\VNAlg)$ if and only if $U'U^{-1} = \alpha_0 \id + \alpha_1\JJ$ with $\alpha_0,\alpha_1 \in\rr$ so that $\alpha_0^2 + \alpha_1^2 = 1$. But operators of this type are precisely the unitary operators in $\Center{\VNAlg}$. 

In the second case, in which $U$ and $U'$ anticommute with $\JJ$, we have that $S\vv = (U_{\cc_{\uI}}\vv)(-\uJ)$ and $S'\vv = (U_{\cc_{\uI}}'\vv)(-\uJ)$ for $\vv\in\hil_{\JJ,\uI}^{+}$ with $U_{\cc_{\uI}} := U|_{\hil_{\JJ,\uI}^{+}}$ and $U_{\cc_{\uI}} := U|_{\hil_{\JJ,\uI}^{+}}'$. Again Wigner's theorem implies that $U$ and $U'$ induce the same symmetry $h$ if and only if $S' = \alpha S$ with $\alpha\in\cc_{\uI}$ and $|\alpha| = 1$. This is equivalent to
\[
\alpha\id \vv = S^{-1}\circ S'(\vv) = S^{-1}\left((U_{\cc_{\uI}}'\vv)\uJ\right) = U_{\cc_{\uI}}^{-1}\left((U_{\cc_{\uI}}'\vv)\uJ(-\uJ)\right) = U_{\cc_{\uI}}^{-1}U_{\cc_{\uI}}'(\vv)
\]
for all $\vv\in\hil_{\JJ,\uI}^{+}$. Taking the quaternionic linear extension to all of $\hil$, we conclude as before that $U$ and $U'$ induce the same symmetry if and only if $U'U^{-1}\in\Center{\VNAlg}$.

\end{proof}

\section{Reduction of Quaternionic to Complex Relativistic Systems }\label{CStructExist}
We recall that the Poincar\'{e} group is the Lie-group of Minkowski-spacetime symmetries and we recall that a unitary representation of a Lie-group $G$ on a real, complex or quaternionic Hilbert space $\hil$ is a group homomorphism $h:G\to \boundOP(\hil)$ such that the mapping $(g,\vv) \mapsto h(g)\vv$ for $g\in G$ and $\vv\in\hil$ is continuous and such that $h(g)$ is unitary for all $g\in G$. Such representation is called locally faithful, if it is injective in a neighborhood of the neutral element of $G$.

An elementary relativistic system is an elementary quantum system that supports a representation $h$ of the Poincar\'{e} group $\poincare$ viewed as a maximal symmetry group of the system. Since the system should be the realisation of the physical symmetries, $h$ must contain all the information about the variables of the system and since the system should be elementary, the representation of $\poincare$ must be irreducible.

\begin{definition}\label{DefiRES}
A real, complex or quaternionic {\em relativistic elementary system} (RES for short) is an elementary quantum system $\VNAlg$ on a real, complex or quaternionic separable Hilbert space $\hil$ that is equipped with a representation of the Poincar\'{e} group $h:\poincare \to \aut(\QMLat(\VNAlg)), g\mapsto h_g$ which is locally faithful and satisfies the following requirements.
\begin{enumerate}[(i)]
\item\label{RES1} $h$ is irreducible, in the sense that if $E\in\QMLat(\VNAlg)$ then  $h_{g}(E) = E$ for all $g\in\poincare$  implies either $E = 0$ or $E = \id$. 
\item\label{RES2} $h$ is continuous in the sense that $g\mapsto \mu(h_g(E))$ is continuous for every fixed $E\in\QMLat(\VNAlg)$ and every fixed quantum state $\mu$. 
\item\label{RES3} $h$ defines the observables of the system. If we represent the symmetry $h_g$ for any $g\in\poincare$ by a unitary operator $U_g\in\VNAlg$ so that $h_g(E) = U_gEU_g^{-1}$, this condition reads as
\begin{equation}\label{h3Cond}
\QMLat(\VNAlg)\subset (\{U_g: g\in\poincare\}\cup\Center{\VNAlg})''.
\end{equation}
\end{enumerate}
\end{definition}
\begin{remark}
The above notion of relativistic elementary system was introduced in Definition~5.7 of \cite{Moretti:2017}. As the authors point out Remark~5.8, it coincides with the usual definition of relativistic elementary systems in the sense of Wigner if the quantum system is complex.
\end{remark}
\begin{remark}\label{hWD}
We implicitly assume in the above definition that we can represent any Poincar\'{e}-induced symmetry $h_g$ by a unitary operator in $\VNAlg$. This is obviously true for real quantum systems of real-real or real-quaternionic type and for real-induced or proper quaternionic quantum systems due to \Cref{VNT-real,VNT}. For the other cases, this follows from the polar decomposition $\poincare = SL(2,\cc)\ltimes\rr^4$ of the Poincar\'{e} group, which allows to write every $g\in P$ as the product $g = rrbb$, where $r$ is a spatial rotation and $b$ is a boost. Even if $\VNAlg$ is a real quantum system of real-complex type or a complex-induced quaternionic quantum system and $r$ or $b$ are represented by unitary operators $U_r$ and $U_b$ that do not belong to $\VNAlg$, then nevertheless $U_r^2$ and $U_b^2$ belong to $\VNAlg$ and so also their product $U_g = U_r^2U_b^2$ belong to $\VNAlg$. Similarly, if $\VNAlg$ is a complex quantum system and $U_r$ or $U_g$ are complex anti-unitary operators, then $U_r^2$ and $U_b^2$ and in turn also $U_g = U_r^2U_g^2$ are complex linear unitary operators in $\VNAlg$. Cf. also \cite[Remark~5.8]{Moretti:2017}.
\end{remark}

Our aim in this section is to show that any quaternionic RES is the external quaternionification of a complex RES, if the operator $M_U^2$ associated with the squared mass of the system is positive.  In order to construct this operator, we choose a Minkowskian reference frame in Minkowski space time and consider the one-parameter Lie-subgroups  $\mathfrak{p}_{\ell}:\rr\to\poincare, \ell = 0,\ldots,3$ of spacetime-displacements along the four Minkowskian axes. If $h:\poincare \to \boundOP(\hil)$ is a unitary representation of the Poincar\'{e} group such that $h_g(E) = U_gE U_g^{-1}$ for all $E\in\QMLat(\VNAlg)$, where $h: g\to h_g$ is the representation of $\poincare$ on $\aut(\QMLat(\VNAlg))$ in \Cref{DefiRES}, then $U_{\ell}(t) := h(\mathfrak{p}_{\ell}(t)) = U_{h_{\mathfrak{p}_{\ell}(t)}}$ is for any $\ell = 0,\ldots,3$ a strongly continuous group of unitary operators on $\hil$. We define $P_{\ell}$ as the infinitesimal generator of the group $U_{\ell}(t)$, which is a densely defined  anti-selfadjoint operator on $\hil$. The operator associated with the squared mass of the system is the operator
\[
M_U^2 := -P_0^2 + \sum_{\ell}^3 P_{\ell}^2,
\]
and we say that $M_U^2$ is positive if $\langle\vv, M_U^2\vv\rangle \geq 0$ for all $\vv\in \dom(M_U^2)$.

\begin{remark}
The above definition of $M_U^2$ is not very rigorous, in particular because it is not immediate that $\dom(M_U^2) = \bigcap_{\ell=0}^3 \dom(P_{\ell})$ is nonempty. A common core for these operators, which is even dense in $\hil$, is the G\aa rding space. It consists of all vectors $\vv_{f}\in\hil$ generated by
\[
\vv_{f} := \int_{\poincare}f(g)\,U_g(\vv)\,dg\quad\text{with}\quad f\in C_{0}^{\infty}(\poincare,\rr),
\]
where $dg$ denotes the left-invariant Haar measure on $\poincare$. The construction of $M_U^2$ for a real RES in \cite{Moretti:2017} is based on a representation of the Lie algebra of $\poincare$ in terms of operators on the G\aa rding space. We can follow the same procedure in the quaternionic setting. Since the properties of the G\aa rding space depend only on the topology on $\hil$ and not on the field of scalars, we can simply choose $\uI\in\SS$ and consider the quaternionic Hilbert space $\hil$ as a complex Hilbert space $\hil_{\uI}$ over $\cc_{\uI}$ . We then obtain the results concerning well-definedness, density etc. of the G\aa rding space and the existence of a representation of the Lie algebra of $\poincare$ in terms of operators on this space simply by applying the complex results on $\hil_{\uI}$. (This is similar to \cite{Moretti:2017}, where the proofs of the properties of the G\aa rding space in the case $\hil$ is a real Hilbert space consist essentially in applying the results from the complex case to the external complexification of the real space $\hil$.) We do not recall these arguments in detail since this would go beyond the scope of this paper.
\end{remark}
Finally, we recall Theorem~5.11 in \cite{Moretti:2017}, which shows that any real quantum system is of real-complex type, provided that the operator $M_U^2$ associated with the squared mass of the system is positive. 

\begin{theorem}\label{TMP1}
Let $\VNAlg$ be a real RES on a real Hilbert space $\hil$ and let $h:\poincare\to\boundOP(\hil)$ be the locally-faithful strongly-continuous unitary representation of the Poincar\'{e} group on $\hil$. If the operator $M_U^2$ that is associated with the squared mass of the system is positive, then the following facts hold.
\begin{enumerate}[(i)]
\item The von Neumann-algebra $\VNAlg$ is of real-complex type and so $\VNAlg = \boundOP(\hil_{\JJ})$, where $\hil_{\JJ} := (\hil, \langle\cdot,\cdot\rangle_{\JJ})$ is the internal complexification of $\hil$ induced by the unitary and anti-selfadjoint operator $\JJ$ such that $\VNAlg' = \{ a\id + b\JJ: a,b\in\rr\}$, which is unique up-to-sign.
\item The representation $h$ of $\poincare$ is irreducible on $\hil_{\JJ}$ and defines a complex RES that is equivalent to $\VNAlg$.
\item The operator $\JJ$ is a Poincar\'{e} invariant and coincides up to sign with the unitary factor of the polar decomposition of the anti-selfadjoint generator of the group of temporal translations, that is  either $P_0 = \JJ|P_0|$ or  $P_0 = -\JJ|P_0|$.
\end{enumerate}
\end{theorem}
\begin{remark}\label{TMP2}
Observe that this theorem in particular implies that a real RES can neither be of real-real nor of real-quaternionic type.
\end{remark}
The proof of the above theorem requires quite lengthy physical arguments. Instead of replicating them in the quaternionic case in order to show that any quaternionic RES is equivalent to a complex one, we therefore first show that any such system is equivalent to a real RES and then apply the above theorem.
\begin{theorem}
Let $\VNAlg$ be a quaternionic RES on a quaternionic Hilbert space $\hil$ and let $h:\poincare\to\boundOP(\hil)$ be a locally-faithful strongly-continuous unitary representation of the Poincar\'{e} group as in \Cref{DefiRES}. If the operator $M_U^2$ associated with the squared mass of the system is positive, then the following facts hold.
\begin{enumerate}[(i)]
\item\label{al1} The von Neumann-algebra $\VNAlg$ is  complex induced and so $\VNAlg$ is the external quaternionification of $\boundOP(\hil_{\JJ,\uI}^{+})$, where $\uI\in\SS$ and $\JJ$ is the unitary and anti-selfadjoint operator $\JJ$ such that $\VNAlg' = \{ a\id + b\JJ: a,b\in\rr\}$ that is unique up-to-sign.
\item\label{al2} The representation $g$ of $\poincare$ is irreducible on $\hil_{\JJ,\uI}^{+}$ and defines a complex RES that is equivalent to $\VNAlg$.
\item\label{al3} The operator $\JJ$ is a Poincar\'{e} invariant and coincides up to sign with the unitary factor of the polar decomposition of the anti-selfadjoint generator of the group of temporal translations, that is either $P_{0} = \JJ |P_0|$ or $P_{0} = -\JJ |P_0|$. 
\end{enumerate}
\end{theorem}
\begin{proof}
Since the von Neumann-algebra $\VNAlg$ is by definition irreducible, it is due to \Cref{VNT} either real-induced, complex induced or proper quaternionic. 

We assume that $\VNAlg$ is real induced. In this case, there exist unitary and anti-selfadjoint operators $\II,\JJ,\KK := \II\JJ$ in $\boundOP(\hil)$ that anti-commute mutually such that 
\begin{equation}\label{JJUAL}
\VNAlg' = \{a\id + b\II + c\JJ + d\KK : a,b,c,d\in\rr\}.
\end{equation}
The von Neumann-algebra $\VNAlg$ is then the external quaternionification of the real von Neumann-algebra $\VNAlg_{\rr} := \VNAlg|_{\hil_{\rr}} = \boundOP(\hil_{\rr})$ of real-real type, which is obviously irreducible. Hence, $\VNAlg_{\rr}$ is an elementary real quantum system that is logically equivalent to $\VNAlg$. 
If $h$ is the locally faithful representation of the poincare group $\poincare$, then it induces a locally faithful representation $h_{\rr}: g\mapsto h_{\rr,g}$ of $\poincare$ on $\hil_{\rr}$ because $\QMLat(\VNAlg)$ and $\QMLat(\VNAlg_{\rr})$ are isomorphic lattices. Precisely, we have for all $E_{\rr}\in\QMLat(\VNAlg_{\rr})$ and all $g\in\poincare$ that
\begin{equation}\label{hpoin}
h_{\rr,g}(E_{\rr}) = h_{g}(E)|_{\hil_{\rr}},
\end{equation}
where $E$ is the quaternionic linear extension of $E_{\rr}$ to all of $E$ so that $E_{\rr} = E|_{\hil_{\rr}}$.  This representation is obviously irreducible: if $h_{\rr,g}(E_{\rr}) = E_{\rr}$ for all $g\in\poincare$ then \eqref{hpoin} implies for the projection $E\in\QMLat(\VNAlg)$ with $E_{\rr} = E|_{\hil_{\rr}}$ that $h_{g}(E) = E$ for all $g\in\poincare$. Due to the irreducibility of $h$, we find that either $E = 0$ or $E = \id_{\hil}$ and in turn also  either $E_{\rr} = E|_{\hil_{\rr}} = 0$ or $E_{\rr}  = E|_{\hil_{\rr}} = \id_{\hil_{\rr}}$. Similarly, any quantum state $\mu_{\rr}:\QMLat(\VNAlg_{\rr})\to [0,+\infty)$, is equivalent to a quantum state $\mu: \QMLat(\VNAlg)\to[0,+\infty)$ and we have
\[
\mu_{\rr}(E_{\rr}) = \mu(E),\quad\text{if}\quad E_{\rr} = E|_{\hil_{\rr}}.
\]
Due to the continuity of $h$, the mapping $g\mapsto \mu_{\rr}(h_{\rr,g}(E_{\rr})) = \mu(h_{g}(E))$ is continuous for any fixed $E_{\rr}\in\QMLat(\VNAlg_{\rr})$ and any quantum state $\mu_{\rr}$ on $\QMLat(\VNAlg_{\rr})$ and so $h$ is continuous. 

Finally, let $\mathfrak{U}_{\rr} := \{U_{\rr,g}: g\in\poincare\}$ be a set of unitary operators on $\hil_{\rr}$ such that $h_{\rr,g}(E_{\rr}) = U_{\rr,g}E_{\rr} U_{\rr,g}^{-1}$ for all $E_{\rr}\in\QMLat(\VNAlg_{\rr})$. The set 
\[
\frak{U} := \left\{ U_{g}\in\boundOP(\hil):U_{g}|_{\hil_{\rr}}\in\mathfrak{U}_{\rr}\right\}
\]
 of quaternionic linear extensions of operators in $\mathfrak{U}_{\rr}$ is then a set of unitary operators on $\hil$ such that $h_{g}(E) = U_{g} E U_{g}^{-1}$ for all $E\in\QMLat(\VNAlg)$ and all $g\in\poincare$. 

Since the operators $\II$, $\JJ$ and $\KK$ commute with any $U\in\mathfrak{U}$, they belong to $\mathfrak{U}'$. Any operator $A\in\mathfrak{U}''$  commutes therefore with $\II$, $\JJ$ and $\KK$ and is hence the quaternionic linear extension of an operator $A_{\rr}\in\boundOP(\hil_{\rr})$. If $B_{\rr}\in\mathfrak{U}_{\rr}'$, then its quaternionic linear extension $B$ to all of $\hil$ belongs to $\mathfrak{U}'$ and so we have $AB = BA$ for any $A\in\mathfrak{U}''$. Taking the restriction to $\hil_{\rr}$, we find that $A_{\rr}B_{\rr} = B_{\rr}A_{\rr}$ for any $A\in\mathfrak{U}''$ and any $B_{\rr}\in\mathfrak{U}_{\rr}'$ and so 
\[
\QMLat(\VNAlg_{\rr}) = \QMLat(\VNAlg)|_{\hil_{\rr}} \subset \mathfrak{U}''|_{\hil_{\rr}} \subset \mathfrak{U}_{\rr}''.
\]
Altogether, we find that $\VNAlg_{\rr}$ is a real RES of real-real type. The operator $M_{\rr,U}^2$ associated with the squared mass in the real quantum system $\VNAlg_{\rr}$ is the restriction of the operator $M_{U}^2$ associated with the squared mass  in the quaternionic quantum system $\hil$ to $\hil_{\rr}$, that is $M_{\rr,U}^2 = M_{U}^2|_{\hil_{\rr}}$. Moreover $M_{U}^2$ is positive if and only if $M_{\rr,U}^2$ is positive. We conclude from \Cref{TMP1} and \Cref{TMP2} that the quaternionic RES $\VNAlg$ is not real-induced if $M_{U}^2$ is positive.

If $\VNAlg$ is proper quaternionic, that is $\VNAlg = \boundOP(\hil)$, then we can argue similarly. We can consider the real Hilbert space $\hil_{\rr}:= \left(\hil,\langle\cdot,\cdot\rangle_{\rr}\right)$, which we obtain if we restrict the scalar multiplication on $\hil$ to $\rr$ and endow this space with the real scalar product $\langle\vu,\vv\rangle_{\rr}:=\Re \langle\vu,\vv\rangle$. If we consider the operators in $\VNAlg$ as $\rr$-linear operators on $\hil_{\rr}$, we find that $\VNAlg$ is a real Banach-subalgebra of $\boundOP(\hil_{\rr})$. Since both scalar products $\langle\cdot,\cdot\rangle$ and $\langle\cdot,\cdot\rangle_{\rr}$ generate the same topology on $\hil$, the set $\VNAlg$ is strongly closed not only as a sub-algebra of $\boundOP(\hil)$ but also as a sub-algebra of $\boundOP(\hil_{\rr})$ and hence it is a real von Neumann-algebra of operators on $\hil_{\rr}$. 

We show now that $\VNAlg$ is also irreducible as a subset of $\boundOP(\hil_{\rr})$ and hence assume that $K$ is a closed $\rr$-linear subspace of $\hil$ such that $T(K) \subset K$ for all $T\in\VNAlg$. The $\hh$-linear span 
\[
\tilde{K} = \linspan{\hh}K = \{\vv a: \vv\in K, a\in \hh\}
\]
 is then a closed quaternionic linear subspace of $\hil$. Since $T\vv \in K$ for any $\vv\in K$ and any $T\in\VNAlg$, we find $T(\vv a) = (T\vv) a \in\tilde{K}$ for any $\vv a \in \tilde{K}$ and any $T\in\VNAlg$, we find that $\tilde{K}$ is a closed quaternionic linear subspace of $\hil$ such that $T(\tilde{K})\subset \tilde{K}$ for any $T\in\VNAlg$. Since $\VNAlg$ is an irreducible von Neumann-algebra on $\hil$, we conclude that either $\tilde{K} = \{ \vO\}$ or $\tilde{K} = \hil$. In the first case, we immediately conclude $K = \{\vO\}$. In the second case, we have $\hil_{\rr} = \hil = \linspan{\hh}K = \{\vv a: \vv\in K, a\in \hh\}$, but not immediate that $\hil = K$.  Hence, let us assume that $K \neq \hil_{\rr}$. Then there exist $\vv \in K$ and $a\in\hh$ such that $\vv a\notin K$. (Without loss of generality, we can even assume $\| \vv \| = 1$.) The operator $T\vu := \vv a \langle \vv, \vu\rangle$ is then a bounded quaternionic right linear operator on $\hil$ and hence belongs to $\VNAlg$, but it does not leave $K$ invariant as $T\vv = \vv a \notin K$. We conclude that such $\vv$ cannot exist and so $\tilde{K} = \hil$ implies $K = \hil = \hil_{\rr}$. Altogether, we find that if a closed subspace $K$ of $\hil_{\rr}$ satisfies $T(K)\subset K$ for all $T\in\VNAlg$, then either $K = \{\vO\}$ or $K = \hil_{\rr}$. Hence, $\VNAlg$ is an irreducible real von Neuman-subalgebra of $\boundOP(\hil_{\rr})$ and therefore an elementary real quantum system on $\hil_{\rr}$ that is equivalent to the elementary quaternionic quantum system $\VNAlg$ on $\hil$. We denote this quantum system by $\VNAlg_{\rr}$ in order to stress that we consider the operators as $\rr$-linear operators on $\hil_{\rr}$. It is obviously of real-quaternionic type according to the classification in \Cref{VNT-real}.

Let $h: \poincare\to\aut(\QMLat(\VNAlg))$ be the locally faithful representation  satisfying condition \cref{RES1,RES2,RES3} in \Cref{DefiRES}. Since $\QMLat(\VNAlg) = \QMLat(\VNAlg_{\rr})$ and in turn $\aut(\QMLat(\VNAlg)) = \aut(\QMLat(\VNAlg_{\rr}))$, we find that $h$ is also a locally faithful representation of the Poincar\'{e} group on  $\aut(\QMLat(\VNAlg_{\rr}))$. We shall denote $h$ also by $h_{\rr}$, if we consider it as a representation $h:\poincare\to\aut(\QMLat(\VNAlg_{\rr}))$. Since $h$ is irreducible, we find that $h_{\rr,g}(E) = h_{g}(E) =  E$ for all $g\in\poincare$ implies either $E = 0$ or $E = \id$ for all $E \in \QMLat(\hil_{\rr}) = \QMLat(\hil)$ and hence also $h_{\rr}$ is irreducible. It obviously also inherits the property of being continuous from $h$. 

What remains to show in order for $\VNAlg_{\rr}$ to be a real RES is that $h_{\rr}$ determines the observables of $\VNAlg_{\rr}$. Let therefore $U_g\in\boundOP(\hil$) for $g\in\poincare$ be unitary operators so that $h_g(E) = U_gEU_g^{-1}$ and denote $\mathfrak{U}:=\{U_g:g\in\poincare\}$ and let us denote $U_g$ considered as an $\rr$-linear operator on $\hil_{\rr}$ by $U_{\rr,g}$. Then $U_{\rr,g}$ is an $\rr$-linear unitary operator on $\hil_{\rr}$ so that $h_{\rr,g}(E) = U_gEU_g^{-1}$. We denote $\mathfrak{U}_{\rr} = \{U_{\rr,g}:g\in\poincare\}$, that is $\mathfrak{U}_{\rr}$ equals $\mathfrak{U}$ where we consider its elements as operators on $\hil_{\rr}$ instead of $\hil$. Since $\VNAlg_{\rr}$ is a real von Neumann-algebra of real-quaternionic type, we have $\Center{\VNAlg_{\rr}} = \{a\id:a\in\rr\}$ by \Cref{VNT-real} and hence we need to show that
\[
\QMLat(\VNAlg_{\rr}) =\QMLat(\VNAlg) \subset (\mathfrak{U}_{\rr}\cup\Center{\VNAlg_{\rr}})'' =  \mathfrak{U}_{\rr}'',
\]
where $\mathfrak{U}_{\rr}''$ denotes the bicommutant of $\mathfrak{U}_{\rr} = \mathfrak{U}$ in $\boundOP(\hil_{\rr})$. We know that the bicommutant $\mathfrak{U}''$ of $\mathfrak{U}$ in $\boundOP(\hil)$ satisfies 
\[
\QMLat(\VNAlg) \subset (\mathfrak{U}\cup\Center{\VNAlg})'' = \mathfrak{U}'',
\]
as $\Center{\VNAlg} = \{a\id: a\in\rr\}$ by \Cref{VNT} because $\VNAlg$ is a proper quaternionic von Neumann algebra. Moreover, as $\VNAlg = \boundOP(\hil)$, we have $\QMLat(\VNAlg) = \QMLat(\hil) = \{a\id: a\in\rr\}$, cf. \cite[Lemma~5.16]{Moretti:2017}, and so
\[
\{a\id: a\in\rr\} = \QMLat(\VNAlg)' \supset (\mathfrak{U}'')' = \mathfrak{U}' \supset \{a\id:a\in\rr\}.
\]
Hence, $\mathfrak{U}' = \{a\id:a\in\rr\}$.

Let now $\uI,\uJ\in\SS$ with $\uI\perp\uJ$ and set $\uK :=\uI\uJ$. For any quaternionic $a\in\hh$, the operator $M_{a}\vv:= \vv a$ is a bounded $\rr$-linear operator on $\hil_{\rr} = \hil$ and hence it belongs to $\boundOP(\hil_{\rr})$.  Moreover, an arbitrary operator in $\boundOP(\hil_{\rr})$ is quaternionic right-linear if and only if it commutes with $M_{\uI}$, $M_{\uJ}$ and $M_{\uK}$. As $\mathfrak{U}_{\rr} = \mathfrak{U}$ consists of quaternionic linear operators, we find $M_{\uI},M_{\uJ},M_{\uK}\in\mathfrak{U}_{\rr}'$ and so
\begin{equation}\label{Uaps}
\{ M_{a}: a\in\hil\} = \{ a_0\id + a_1M_{\uI} + a_2 M_{\uJ}  + a_3M_{\uK}: a_{\ell}\in\rr\} \subset \mathfrak{U}_{\rr}'.
\end{equation}
If on the other hand $A\in\mathfrak{U}_{\rr}'$, then the operator
\begin{align*}
A_{\hh}\vv :=& A\vv - M_{\uI}AM_{\uI}\vv - M_{\uJ}AM_{\uJ}\vv - M_{\uK}AM_{\uK}\vv \\
=& A\vv - (A(\vv\uI))\uI - (A(\vv\uJ))\uJ - (A(\vv\uK)\uK)
\end{align*}
obviously also belongs to $\mathfrak{U}_{\rr}'$ because it consists of the sum of compositions of operators in $\mathfrak{U}_{\rr}'$. We moreover have
\begin{align*}
A_{\hh}(\vv\uI) = & A(\vv\uI) - (A(\vv\uI^2))\uI - (A(\vv\uI\uJ))\uJ - (A(\vv\uI\uK)\uK)\\
 = &\left( -A(\vv\uI)\uI + A(\vv) - (A(\vv\uK))\uK- (A(\vv\uJ)\uJ)\right)\uI = (A_{\hh}\vv)\uI
\end{align*}
and similarly also $A_{\hh}(\vv\uJ) = (A_{\hh}\vv)\uJ$ and $A_{\hh}(\vv\uK) = (A_{\hh}\vv)\uK$. Hence, $A_{\hh}$ is a quaternionic right linear operator in $\mathfrak{U}_{\rr}'$. Therefore it belongs to $\mathfrak{U}'$ and we conclude
\[
A_{\hh}(\vv\uI) = \vv a_0
\]
with $a_{0}\in\rr$. The operators $AM_{\uI}$, $AM_{\uJ}$ and $AM_{\uK}$ also belong to $\mathfrak{U}_{\rr}'$ because they are compositions of operators in $\mathfrak{U}_{\rr}'$. By the above argument, there exist real numbers $a_1$, $a_2$, and $a_3$ such that $(AM_{\uI})_{\hh} = a_1\id$, $(AM_{\uJ}) = a_2\id$, and $(AM_{\uK})_{\hh} = a_3\id$. Straightforward computations show that
\begin{align*}
A_{\hh}\vv = A(\vv) - (A(\vv\uI))\uI - (A(\vv\uJ))\uJ - (A(\vv\uK))\uK =& \vv a_0\\
((AM_{\uI})_{\hh}\vv)(-\uI) = - A(\vv\uI)\uI + A(\vv) + (A(\vv\uK))\uK + (A(\vv\uJ))\uJ =& -\vv a_1\uI\\
((AM_{\uJ})_{\hh}\vv)(-\uJ) = -(A(\vv\uJ))\uJ + (A(\vv\uK))\uK + A(\vv) + (A(\vv\uI))\uI =& -\vv a_2\uJ\\
((AM_{\uK})_{\hh}\vv)(-\uK) =  - A(\vv\uK)\uK + (A(\vv\uJ))\uJ + (A(\vv\uI))\uI + A(\vv) =& -\vv a_3\uK.
\end{align*}
If we add these four equations, we are left with
\[
4 A(\vv) = \vv a_0 - \vv a_1\uI - \vv a_2\uJ - \vv a_3\uK
\]
and hence $A = M_{a}$ with $ a = \frac{1}{4} (a_0  - a_1 \uI - a_2 \uJ - a_3 \uK)$. Hence, the relation in \eqref{Uaps} is actually an equality and $\mathfrak{U}_{\rr}' = \{ M_a: a\in\hh\}$. An operator $T\in\boundOP(\hil_{\rr})$ commutes with $M_a$ if and only if
\[
T(\vv a) = TM_a\vv = M_a T\vv = (T\vv)a
\]
and hence the set of operators that commute with  $M_a$ for any $a\in\hh$ is exactly the set of quaternionic linear operators in $\boundOP(\hil_{\rr})$. We conclude $\mathfrak{U}_{\rr}'' = \boundOP(\hil) = \mathfrak{U}''$ and so in particular $\QMLat(\VNAlg_{\rr}) = \QMLat(\VNAlg) \subset \mathfrak{U}'' = \mathfrak{U}_{\rr}''$.

Althogether, we find that $\VNAlg_{\rr}$ is a real RES of real-quaternionic type on $\hil_{\rr}$. However, if $M_{U}^2$ is the operator associated with the squared mass of the system in $\VNAlg$, then $M_{U}^2$ is also the operator associated with the squared mass of the system in $\VNAlg_{\rr}$ and it is positive on $\hil$ if and only if it is positive as an operator on $\hil_{\rr}$. Since the real RES cannot be of real-quaternionic type if $M_U^2\geq 0$, we conclude that positivity of $M_U^2$ implies that $\VNAlg$ is not proper quaternionic.

The von Neumann-algebra $\VNAlg$ must hence be complex induced. If $\JJ$ is the unitary anti-selfadjoint operator on $\hil$ such that $\VNAlg' = \{ a\id + \JJ b:a,b\in\rr\}$ and $\uI\in\SS$, then $\VNAlg$ is the external quaternionification of the $\cc_{\uI}$-complex von Neumann-algebra $\VNAlg_{\uI}:= \boundOP(\hil_{\JJ,\uI}^{+})$, which is obviously irreducible, and $\QMLat(\VNAlg)$ is the external quaternionification of and hence isomorphic to $\QMLat(\VNAlg_{\uI}) = \QMLat(\hil_{\JJ,\uI}^{+})$. The representation $h:\poincare\to\aut(\QMLat(\VNAlg))$ induces also in this case a representation $h_{\uI}:\poincare\to\aut(\QMLat(\VNAlg_{\uI}))$ of the Poincar\'{e} group, namely
\[
h_{\uI,g}(E_{\uI}) := h(E)|_{\hil_{\JJ,\uI}^{+}}\quad\text{if }E_{\uI} = E|_{\hil_{\JJ,\uI}^{+}}.
\]
The same arguments that we applied in the real-induced and proper quaternionic case show that $h_{\uI}$ inherits irreducibility and continuity from $h$. Finally, let $U_{\uI,g}$ for $g\in\poincare$ be a unitary operator on $\hil_{\JJ,\uI}^{+}$ so that  $h_{\uI,g}(E_{\uI}) = U_{\uI,g}E_{\uI} U_{\uI,g}^{-1}$ for all $E_{\uI}\in\QMLat(\VNAlg_{\uI})$ and set $\mathfrak{U}_{\uI} = \{ U_{\uI,g}:g\in\poincare\}$. The quaternionic linear extension $U_g$ of $U_{g,\uI}$ to all of $\hil$ is then a unitary operator on $\hil_{\JJ,\uI}^{+}$ such that $h_g(E) = U_{g}EU_{g}^{-1}$ for any $E\in\QMLat(\VNAlg)$. Since $\VNAlg$ is a complex-induced quaternionic RES, we have 
\[
\QMLat(\VNAlg)\subset (\mathfrak{U} \cup \Center{\VNAlg})'' = (\mathfrak{U} \cup \{\JJ\})''
\]
because $\Center{\VNAlg} = \{ a\id + b\JJ : a,b\in\rr\}$. An operator $A\in\boundOP(\hil)$ belongs to $(\mathfrak{U}\cap\{\JJ\})'$ if and only if it commutes with every operator $U\in\mathfrak{U}$ and with the operator $\JJ$. This is equivalent to $A$ being the quaternionic linear extension of an operator $A_{\uI} = A|_{\hil_{\JJ,\uI}^{+}}$ in $\boundOP(\hil_{\JJ,\uI}^{+})$ that commutes with the restriction $U_{\uI} = U|_{\hil_{\JJ,\uI}^{+}}$ of any $U\in\mathfrak{U}$, in other words to  $A_{\uI}$ being an element of $\mathfrak{U}_{\uI}'$. Therefore 
\begin{equation}\label{RRaSL}
(\mathfrak{U}\cup\Center{\VNAlg})' = (\mathfrak{U}_{\uI}')_{\hh} := \left\{ A\in\boundOP(\hil): A|_{\hil_{\JJ,\uI}^{+}}\in\mathfrak{U}_{\uI}'\right\}.
\end{equation}
In particular $\JJ\in (\mathfrak{U}\cup\Center{\VNAlg})'$ and so also any operator $A\in (\mathfrak{U}\cup\Center{\VNAlg})''$ is the quaternionic linear extension of an operator in $\boundOP(\hil_{\JJ,\uI}^{+})$. Therefore
\begin{equation}\label{AAMK2}
(\mathfrak{U}\cup\Center{\VNAlg})'' = \{ A\in\boundOP(\hil): A|_{\hil_{\JJ,\uI}^{+}}\in\mathfrak{U}_{\uI}''\}
\end{equation}
because two operators $A,B\in\boundOP(\hil)$ so that $A_{\uI} = A|_{\hil_{\JJ,\uI}^{+}}$ and $B_{\uI} = B|_{\hil_{\JJ,\uI}^{+}}$ belong to $\boundOP(\hil_{\JJ,\uI}^{+})$ commute if and only if $A_{\uI}$ and $B_{\uI}$ commute. Combining \eqref{RRaSL} and \eqref{AAMK2} and taking the restrictions to $\hil_{\JJ,\uI}^{+}$, we obtain
\[
\QMLat(\VNAlg_{\uI}) = \left\{ E|_{\hil_{\JJ,\uI}^{+}} : E \in\QMLat(\VNAlg)\right\} \subset \left\{ A|_{\hil_{\JJ,\uI}^{+}}: A\in(\mathfrak{U}\cup\{\JJ\})''\right\} = \mathfrak{U}_{\uI}''
\]
and so $\VNAlg_{\uI}$ is a complex RES, which concludes the proofs of \cref{al1} and \cref{al2}. 

The operator $\JJ$ is Poincar\'{e} invariant because it commutes with every operator in $\VNAlg$ and so in particular with $U_g$ for any $g\in\poincare$. Hence $U_g \JJ U_g^{-1} = \JJ U_gU_g^{-1} = \JJ$. Finally, if $P_0$ is the infinitesimal generator of the unitary group of temporal translations in $\VNAlg$ and $P_0 = J |P_0|$ is the polar decomposition of $P_0$, then $P_{0,\uI} := P_{0}|_{\hil_{\JJ,\uI}^{+}}$ is the infinitesimal generator of the unitary group of temporal translations in $\VNAlg_{\uI}$ and its polar decomposition is given by $P_{0,\uI} = J_{\uI} |P_{0,\uI}|$ with $J_{\uI} := J|_{\hil_{\JJ,\uI}^{+}}$ and $|P_{0,\uI}| = |P_{0}||_{\hil_{\JJ,\uI}^{+}}$. Since $\VNAlg_{\uI}$ is a complex RES, we have by Theorem~4.3 in \cite{Moretti:2017} however that either 
\[
J_{\uI} = \uI\id_{\hil_{\JJ,\uI}^{+}} = \JJ|_{\hil_{\JJ,\uI}^{+}}\quad \text{or} \quad J_{\uI} = -\uI\id_{\hil_{\JJ,\uI}^{+}} = -\JJ|_{\hil_{\JJ,\uI}^{+}}
\]
and so in turn $J = \JJ$ or $J = -\JJ$. 

\end{proof}
\begin{remark}
If we set $\tilde{H} := cP_0$, where $c$ is the speed of light, then the polar decomposition of $\tilde{H}$ is $\tilde{H} = \JJ H$ with $ H = |\tilde{H}| = c|P_0|$. As pointed on in \cite[p.~34]{Moretti:2017}, the operator $H$ can then be interpreted as a the energy operator, that is the Hamiltonian, of the system.
\end{remark}

\section{The Fundamental Logical Mistake in Quaternionic Quantum Mechanics}
The above results suggest that quaternionic quantum mechanics is actually equivalent to classical complex quantum mechanics despite the fact that researchers in this field claimed the incompatibility of the two theories \cite{Adler:1995}. This misconception is caused by one fundamental logical mistake that was made in quaternionic quantum mechanics  from its very beginning in the foundational paper \cite{Finkelstein:1962}. It is the assumption that there exists a privileged left multiplication that is compatible with the physical theory. 

We recall that the initial motivation for developing a quaternionic version of quantum mechanics was the fact that the propositional calculus of a quantum system consists of an orthomodular lattice, which can be realised as a lattice of subspaces on such Hilbert space \cite{Birkhoff:1936}. A left multiplication is however not determined by the calculus. Indeed, the argument in \cite{Finkelstein:1962} for the existence of a privileged left multiplication is not correct. The authors argue that the physical properties of a quantum system should not depend on the concrete realisation of the number field that one uses so that the system should be invariant under automorphisms of the number field. If one considers a quantum system on a Hilbert space $\hil$ over $\mathbb{F}$, where $\mathbb{F}$ is one of the fields $\rr$, $\cc$ or $\hh$, and $\phi$ is an automorphism of the $\mathbb{F}$, then there exists a an associated co-unitary transformation of $\hil$, that is an additive mapping $U_{\phi}:\hil\to\hil$ such that
 \begin{equation}\label{RMQQ}
U_{\phi}(\vv a) = U_{\phi}(\vv)\phi(a) \qquad\text{and}\qquad  \langle(U_{\phi}(\vv),U_{\phi}(\vu)\rangle= \phi(\langle\vv,\vu\rangle).
 \end{equation}
   All laws of the quantum system must be covariant under the transformation $U_{\phi}$. Any automorphism $\phi$ of $\hh$ is however of the form $\phi(x) = hxh^{-1}$ with $h\in\hh$ and $|h| = 1$. The authors hence argue that the mapping $\vv\mapsto U_{\phi}(\vv) h$ is then quaternionic right linear because
 \[
 U_{\phi}(\vv a)h = U_{\phi}(\vv)\phi(a)h = U_{\phi}(\vv)hah^{-1} h = U_{\phi}(\vv)h a
\]
and hence they define a left multiplication on $\hil$ via
\[
h\vv := U_{\phi}(\vv)h.
\]
However, the co-unitary mapping $U_{\phi}$ is not well-defined. Any unitary operator $U$ that induces a symmetry of the system defines via $\vv \mapsto U\vv h^{-1}$ a co-unitary mapping that satisfies \eqref{RMQQ}. If we choose the same symmetry $U$ for any automorphism $\phi$, then
\[
h\vv = U_{\phi}(\vv)h = U\vv h^{-1}h = U\vv
\]
for any $\phi$ so that $h\vv$ is independent of $h$. We could even choose $U = \id$, so that we define $h\vv = \vv$, which is obviously nonsense.

We conclude this chapter with an examples of a seeming inconsistency between the complex and the quaternionic theory that arises from the assumption of the existence of a left multiplication. This inconsistency is however resolved if only the existence of a compatible unitary and anti-selfadjoint operator $\JJ$ that commutes with any observable and the unitary group of time translations is assumed. We point out that also other discrepancies such as the difficulty of defining a proper momentum operator, showing Heisenberg's uncertainty principle or developing a framework for composite systems can be resolved if the existence of such operator $\JJ$ is assumed. In order to explain the discrepancy we want to resolve, we quickly recall the main concepts of quaternionic quantum mechanics introduced in \cite{Adler:1995}. We shall furthermore adopt the Bra-Ket-notation used by physicists in order to make the comparison for the reader easier.

Quaternionic quantum mechanics is in \cite{Adler:1995} formulated on an abstract Hilbert space $\hil$ over the quaternions $\hh = \{ x_0 + \uI x_1 + \uJ x_2 + \uK x_3: x_{\ell}\in\rr\}$  consisting of ket-vectors $|f\rangle$, where the same vector considered as an element of the dual space via the Riesz-representation theorem is denoted by the bra-vector $\langle f|$. (The dimension of $\hil$ is assumed to be greater than two in order for Wigner's theorem to hold.) The scalar product on $\hil$ is denoted by $\langle g| f\rangle$ and (pure) states of the system correspond to one-dimensional rays of the form $|f\omega\rangle$ with $\omega \in\hh$. Observables are self-adjoint operators $A$ on $\hil$, that is $A^{\dagger} = A$ where $A^{\dagger}$ denotes the adjoint. Any observable $A$ has a representation of the form $A = \sum |a\rangle a \langle a|$ in terms of an orthonormal eigenbasis of eigenvector $|a\rangle$ associated with eigenvalues $a$ and the expectation value of measuring the observable $A$ if the system is in the state $|f\rangle$ is given by $\langle f| A |a\rangle = \sum_{a} \langle f|a\rangle a \langle a| f\rangle  = \sum a |\langle a|f\rangle|^2$. Symmetries of the systems are unitary operators on $\hil$. A continuous one-parameter group of symmetries is of the form $U(t) = e^{tA}$, where $A$ is an anti-selfadjoint operator satisfying $A^{\dagger} = -A$. The eigenvalues of such $A$ are (equivalence classes $[a]$ of) purely imaginary quaternions and the eigenspaces associated with different eigenvalues are mutually orthogonal. Choosing for any eigenvalue the representative $a$ that belongs to the upper complex halfplane $\cc_{\uI}^{\geq}$, we can find a representation of $A$ of the form 
\[
A = \sum |a\rangle a \langle a| = \sum |a\rangle \uI |a|\langle a|.
\]
We can furthermore set $\II_{A} := \sum |a\rangle \uI \langle a|$ and $|A| = \sum |a\rangle |a| \langle a| $ and find $A = \II_{A} |A|$, which corresponds to the polar decomposition of $A$. 
 If we set $\JJ_{A} := \sum |a\rangle \uJ \langle a|$ and $\KK_{A} := \sum |a\rangle \uK \langle a|$, then we obtain a left multiplication that commutes with $|A|$ and we can write any state $|f\rangle$ in terms of its four real components 
 \[
 |f\rangle = |f_0\rangle + \II_{A} |f_1\rangle + \JJ_{A}|f_2\rangle + \KK_{A} |f_3\rangle =  |f_0\rangle + |f_1\rangle \uI+ |f_2\rangle\uJ + |f_3\rangle\uK.
 \] 
 (Note however that only $\II_{A}$ is determined by $A$, the operators $\JJ_A$ and $\KK_A$ depend on the eigenbasis of $A$ that we choose.) 
 
 We consider the position operator $X$ that has a (continuous) eigenbasis $|x\rangle$ such that $X|x\rangle = |x\rangle x$ on $\hil$ and such that
 \[
 \id = \int dx^3 \, |x\rangle \langle x|.
 \] 
 Adler uses this position operator in order to define the quaternion-valued wave function associated with a state $|f\rangle$ as
 \[
 f(x) = \langle x | f \rangle
 \]
and finds that
 \[
 \langle g,f\rangle = \int dx^3\, \langle g|x\rangle \langle x| f\rangle = \int dx^3\, \overline{g(x)}f (x).
 \]
 We assume further more the left multiplication on $\hil$ to be the left multiplication induced by 
 \begin{equation}\label{LMultiAdler}
 \II := \II_{x} = \int |x\rangle \uI \langle x|, \quad \JJ := \JJ_{x} = \int |x\rangle \uJ \langle x|,\quad\text{and}\quad \KK := \KK_{x} = \int |x\rangle \uK \langle x|
 \end{equation}
  and find that this corresponds to the natural pointwise multiplication of the wave function, that is $(af)(x) = \langle  x|af\rangle = a(f(x))$ for all $a\in\hh$. We once more stress that this choice is however made by the author because it is convenient when working with wave functions, but it is not determined by the physical system and hence does not carry physical information.

 The time evolution of the system is described by symmetries $U(t,  \delta t)$ mapping the state of the system at time $t$ to the state of the system at time $t+\delta t$.  We define $-H$ to be the first coefficient of the Taylor series expansion of $U(t,\delta t) = \id + \delta t (-H)$, that is
 \[
 U(t,\delta t)|f(t)\rangle = \id |f(t)  - H |f(t)\rangle + o(\delta t^2)
 \]
 Together with 
 \[
|f(t + \delta t)\rangle =  | f(t)\rangle + \delta t  \frac{\partial}{\partial t}|f(t)\rangle + o(\delta t^2),
\]
we find the dynamics of the system being described by the  Schr\"{o}dinger equation
\begin{equation}\label{Hamiltonian}
 \frac{\partial}{\partial t} |f(t)\rangle = - H(t)|f(t)\rangle.
\end{equation}
An equation of this type is however only possible as long as we stick with a certain representative of the ray that describes the physical state. If we consider a general ray representative $|f(t)\omega_{f}(t)\rangle$ with a quaternionic phase $\omega_f(t) \in \hh$ and $|\omega_f(t)| = 1$, then the corresponding dynamical equation is
\begin{equation}\label{QuatPhase}
\frac{\partial}{\partial t}|f(t)\omega_{f}(t)\rangle =  - H(t)|f(t)\omega_f(t)\rangle  + |f(t)\omega_f(t)\rangle h_f(t) 
\end{equation}
with 
\[
h_f(t) = \omega_f(t)\omega_f'(t) = \overline{\omega_f(t)}\omega_f'(t).
\]
If we differentiate $\overline{\omega_f(t)}\omega_f(t) = 1$, we find that $h_f(t) = \overline{\omega_f(t)}\omega_f'(t) = -\overline{\omega_f'(t)}\omega_f(t) = - \overline{h_f(t)}$ and hence $h_f(t)$ is a purely imaginary quaternion. In contrast to the complex case, it can hence not be commuted with $\omega_f(t)$ in order to integrate it into a modified Hamiltonian and in order to obtain again an equation of the form \eqref{Hamiltonian}. 

Adler finally argues in \cite[pp. 41]{Adler:1995} that complex and quaternionic quantum mechanics are inequivalent because they have different transition probabilities and vectors that represent the same state in quaternionic quantum mechanics can be orthogonal and hence represent different states in complex quantum mechanics. He writes the wave function 
\[
f(x) = \langle x, f\rangle =  f_0(x) + f_1(x) \uI + f_2(x) \uJ + f_3(x)\uK
\]
of a state $|f(x)\rangle$ in position coordinates, in terms of two symplectic components as
\[
f(x) = F_1(x) + \uJ F_2(x)
\]
with 
\[
 F_1(x) = f_0(x) + f_1(x) \qquad F_2(x) = f_2(x) - f_3(x)\uI.
\]
He then assumes the Hamiltonian $H$ to be written in terms of real components 
\[
H = H_0 + H_1\II + H_2\JJ+ H_3\KK = H_0 + \uI H_1 + \uJ H_2 + \uK H_3.
\] 
(Note that this implicitly assumes that $H$ is the quaternionic linear extension of an $\rr$-linear operator from the real component space to $\hil$.) The coordinate representation of $H$ in the position basis is then
\[
H_{\ell}(x) \langle x| = \langle x | H_{\ell} \qquad \ell \in\{0,\ldots,3\}.
\]
If we define $\mathcal{H}_{1}(x) = H_{0}(x) + \uI H_1(x)$ and $\mathcal{H}_{2}(x) := H_2(x) - \uI H_3(x)$ and the matrix-valued function
\[
\mathcal{H}(x) := \left(\begin{array}{cc} \mathcal{H}_{1}(x) &  - \overline{\mathcal{H}_2(x)} \\ \mathcal{H}_2(x) & \overline{\mathcal{H}_1(x)}\end{array}\right),
\]
then the Schrödinger equation can be rewritten as
\[
\frac{\partial}{\partial t} \begin{pmatrix} F_1(x,t) \\ F_2(x,t) \end{pmatrix} = \mathcal{H}(x) \begin{pmatrix} F_1(x,t) \\ F_2(x,t) \end{pmatrix}.
\]
Endowing the space of all component functions with the usual scalar product on the direct sum $L^2(\cc_{\uI})\oplus L^2(\cc_{\uI})$, namely
\begin{equation}\label{ComplScal}
\langle f(x), g(x)\rangle_{\cc} := \int  \overline{F_1(x)}G_1(x) + \overline{F_2(x)}G_2(x)\, dx^3, 
\end{equation}
we find that this describes a complex quantum system since the only imaginary unit that appears in the above equations is the unit $\uI$. The units $\uJ$ and $\uK$ disappeared. 

Adler argues now that this system is inequivalent to the original quaternionic linear system. The quaternionic scalar product of the states $|f\rangle$ and $|g\rangle$ resp. of their wave functions $f(x)$ and $g(x)$ is
\[
\langle f(x), g(x)\rangle = \int \overline{f(x)} g(x) \,dx^3 = \langle f(x), g(x)\rangle_{\cc} + \uJ \langle f(x),g(x)\rangle_{\mathcal{S}},
\]
where $\langle f(x), g(x)\rangle_{\cc}$ is as in \eqref{ComplScal} and $\langle f(x),g(x)\rangle_{\mathcal{S}}$ is the symplectic scalar product 
\[
\langle f(x),g(x)\rangle_{\mathcal{S}} := \int F_1(x) G_2(x) - F_2(x)G_1(x)\,dx^3.
\]
Therefore the transition probabilities are 
\[
| \langle f(x), g(x)\rangle_{\cc}|^2
\]
 in the complex, but
\[
|\langle f(x), g(x)\rangle|^2 = | \langle f(x), g(x)\rangle_{\cc} |^2+  |\langle f(x),g(x)\rangle_{\mathcal{S}}|^2
\] in the quaternionic system. Furthermore, an orthonormal basis  $|h_{\ell}\rangle$ of the quaternionic system is not complete in the complex one. Instead one has to extend it and consider the set $|h_{\ell}\rangle, |h_{\ell}\uJ\rangle$ in order to obtain an eigenbasis of the complex system. In particular, $|h_{\ell}\rangle$  and $|h_{\ell}\uJ\rangle$ belong to the same ray and hence they represent the same state in the quaternionic system. However, they are orthogonal in the complex system and hence represent different states (in general even associated with different eigenvalues) in the complex system. 

Adler was obviously right, these two systems are not equivalent. However, as we know from our preceding analysis, he compared the wrong systems. The quaternionic system cannot be equivalent to the system that we obtain if we consider the entire quaternionic Hilbert space $\hil$ as the complex system  by simply taking the $\cc_{\uI}$-linear part of the quaternionic scalar product. As pointed out by Adler, this introduces new orthogonality relations so that the vectors $|h_{\ell}\rangle$ and $|h_{\ell}\uJ\rangle$, which describe the same state in the quaternionic system, correspond to different states in the complex one. The phase space $\hil_{\cc}$ of a complex system that is equivalent to a quaternionic one can hence not be the entire quaternionic Hilbert space $\hil$. It must be a subspace that does not contain $|h_{\ell}\uJ\rangle$ whenever $|h_{\ell}\rangle\in\hil_{\cc}$. The natural candidate for such space consists of all vectors that have complex wave functions such that $F_2(x) \equiv 0$. It is however not clear whether any state has a representative in this space, whether it is invariant under time translations etc. The reason for this is that the left multiplication \eqref{LMultiAdler}, which determines the component functions of the wave function and in turn this set of vectors, is not motivated by physical arguments but by the fact that it is convenient to work with.

 Instead we have to consider a complex system on the complex subspace
 \[
\hil_{\JJ_{H},\uI}^{+} := \{ \vv \in\hil_{H}: \JJ_{H}\vv = \vv \uI\},
\]
where $\JJ_{H}$ is the unitary anti-selfadjoint operator that appears in the polar decomposition of the anti-selfadjoint Hamiltonian $H$.  From our previous discussion we already know that this system is equivalent to the system on the entire quaternionic space  if $\JJ_{H}$ commutes with any observable. The scalar product of any two vectors in this subspace is moreover naturally complex, so that we do not have to remove information by discarding the symplectic part in the quaternionic scalar product and hence no additional orthogonality relations are introduced.

We conclude by observing that quaternionic quantum mechanics as developed in \cite{Adler:1995} is not actually a quaternionic theory. A proper quaternionic theory would consider equivalence classes of eigenvalues---spectral spheres as they are determined by the $S$-eigenvalue operator---and then admit arbitrary vectors in the associated eigenspaces to represent states of the system. In particular it would admit arbitrary quaternionic phases in \eqref{QuatPhase}. The current version of quaternionic quantum mechanics however chooses eigenvectors associated with eigenvalues in one fixed complex plane $\cc_{\uI}$ and then only works with these vectors. (This was of course also done because a proper theory of quaternionic linear operators and in particular the definition of the $S$-spectrum were not known when the theory was developed.) From our perspective, this does however correspond to considering the quaternionic system actually as a complex system over the complex field $\cc_{\uI}$ and unconsciously working only in the Hilbert space $\hil_{\JJ_{H},\uI}^{+}$.

\bibliography{/Users/jonathan/Dropbox/Bibliothek}
\bibliographystyle{plain}

\textsc{Politecnico di Milano, Dipartimento di Matematica, Via E. Bonardi, 9, 20133 Milano, Italy}\\
E-Mail: gantner.jonathan@gmail.com

\end{document}